\documentclass[11pt]{article}
\usepackage{amsmath,amssymb,amsfonts,amsthm}
\usepackage{graphicx}
\usepackage[english]{babel}
\usepackage{enumerate}
\usepackage{multicol}
\usepackage{times}
\usepackage[T1]{fontenc}
\usepackage{multimedia}
\usepackage[margin=1in]{geometry}
\usepackage{color}
\usepackage{verbatim}

\newcommand{\RR}{\mathbb{R}}
\newcommand{\eps}{\varepsilon}
\DeclareMathOperator\erf{erf}
\newcommand{\half}{\mbox{$\frac{1}{2}$}}

\newtheorem{theorem}{Theorem}[section]
\newtheorem{lemma}[theorem]{Lemma}
\newtheorem{remark}[theorem]{Remark}
\newtheorem{corollary}[theorem]{Corollary}

\theoremstyle{definition}
\newtheorem{definition}[theorem]{Definition}

\numberwithin{equation}{section}

\title{A PDE Approach to the Prediction of a Binary Sequence with Advice from Two
History-Dependent Experts\footnote{This research was partially supported
by NSF grant DMS-1311833.}}

\author{Nadejda Drenska\footnote{Department Mathematics, University of
Minnesota, ndrenska@umn.edu} \ and Robert V. Kohn\footnote{Courant Institute
of Mathematical Sciences, New York University, kohn@cims.nyu.edu}
}

\date{\today}

\begin{document}

\maketitle

\begin{sloppypar} 

\begin{abstract}
The prediction of a binary sequence is a classic example of online machine learning. We like to call it
the ``stock prediction problem,'' viewing the sequence as the price history of a stock that goes up or down one
unit at each time step. In this problem, an investor has access to the predictions of two or more ``experts,'' and strives
to minimize her final-time regret with respect to the best-performing expert. Probability plays no role; rather, the
market is assumed to be adversarial. We consider the case when there are \emph{two history-dependent}
experts, whose predictions are determined by the $d$ most recent stock moves.  Focusing on an appropriate
continuum limit and using methods from optimal control, graph theory, and partial differential equations, we discuss
strategies for the investor and the adversarial market, and we determine associated upper and lower bounds
for the investor's final-time regret. When $ d \leq 4$ our upper and lower bounds coalesce, so the
proposed strategies are asymptotically optimal. Compared to other recent applications of partial differential equations
to prediction, ours has a new element: there are two timescales, since the recent history changes at every step whereas
regret accumulates more slowly.
\end{abstract}

\section{Introduction} \label{INT}

Prediction with expert advice is an area of online machine learning
\cite{CBL}, in which an agent has access to several experts' predictions
and uses them to make a prediction of her own. Probabilistic modeling
is not involved; instead, the agent's goal is to ``minimize regret'' -- i.e.
to minimize her worst-case shortfall with respect to the (retrospectively)
best-performing expert.

Much of the vast literature in this area explores easily-implemented
prediction strategies, assessing their performance on broad classes
of prediction problems. Our focus is different: we examine a special
example -- the prediction of a binary sequence using guidance from two history-dependent
experts (described informally below, and more formally in
Section \ref{GS}) -- looking for prediction strategies that
take into account the character of the example. Our work is
interesting because:

\begin{enumerate}
\item[(a)] we show how considering a scaled version of the problem
facilitates the use of PDE methods;
\item[(b)] we show how the study of history-dependent experts leads
naturally to the use of graph-theoretic methods; and
\item[(c)] we identify \emph{asymptotically optimal} strategies when the
experts use only the most recent $d \leq 4$ stock moves.
\end{enumerate}

(We only consider two experts, and when $d \geq 5$ our strategies might
not be optimal. Subsequent to the work presented here, the first author and
J.~Calder have obtained further results, identifying asymptotically optimal
strategies for any value of $d$ and any number of history-dependent
experts \cite{CD}. The methods used in that work are somewhat different
from -- and complementary to -- those of the present paper.)

The idea that scaling should facilitate the use of PDE methods is
not unprecedented. Our prediction problem can be viewed
as a two-person game, and the PDE describes the value of the game in
an appropriate scaling limit. Other two-person games leading similarly
to nonlinear parabolic equations have been studied, for example,
in \cite{APSS,AS1,KS1,KS2,LM,PS2,PS1}. Much of this literature uses viscosity-solution
techniques, which are required when the solution of the relevant PDE is
not smooth. In this paper we don't need viscosity solutions, since our
PDE has a smooth solution; as a result, our arguments are rather
elementary, using what a control theorist would call ``verification
arguments.''

While the standard name for our problem is the ``prediction of a binary sequence using
expert advice,'' we like to call it the ``stock prediction problem,''
viewing the sequence as the price history of a stock that goes up or down one
unit at each time step. The case with two {\it static} experts (one who
always predicts the stock will go up, the other who always predicts
the stock will go down) has a long history, going back at least to
Thomas Cover's 1965 paper \cite{Cover}. A concise treatment can be
found in Section 8.4 of \cite{CBL}. The optimal strategies can be
determined using dynamic programming, and the predictor's worst-case
regret has an explicit formula, involving the expected value of a
function of many Bernoulli random variables. When the number of steps
is large, the explicit formula can be evaluated using the central
limit theorem. An alternative PDE-based perspective on this continuum
limit was developed by Kangping Zhu in \cite{Zhu}; roughly speaking,
it takes the continuum limit of the dynamic programming principle
rather than the explicit solution formula.

In this paper we consider two {\it history-dependent} experts, whose
predictions depend on the most recent $d \geq 1$ stock movements. Our
overall approach is similar to that of \cite{Zhu} -- which is only
natural, since when $d=0$ the experts are static rather than
history-dependent. When $d \geq 1$, however, the problem is fundamentally
different from those considered by Cover \cite{Cover} and Zhu \cite{Zhu}.
Briefly: regret accumulates slowly, while the recent history changes
quickly. In more detail: if the stock price moves by $\pm 1$ at each
time step then the worst-case regret should be of order $\sqrt{N}$
after $N$ time steps, by arguments related to Blackwell's approachability
theorem (see e.g. Chapter 2 of \cite{CBL}). So the regret increases
about $1/\sqrt{N}$ per time step, which is small (since we are
interested in the asymptotic behavior as $N \rightarrow \infty$).
The recent history, by contrast, is different at each time step.
Thus for history-dependent experts the problem has two well-separated
time scales. In our analysis, the fast timescale is handled using
graph-theoretic methods, while the slow timescale is handled using
methods analogous to those of \cite{Zhu}.

Here is an informal description of our version of the stock prediction
problem. Consider a stock whose price goes up or down one unit at each
time step. An investor can invest (at each timestep) in $-1\leq f \leq 1$
units of stock. There are two experts, whose investment choices are
determined by two functions (call them $q$ and $r$) of the most recent $d$
stock moves. The clock stops after $N$ time steps, and the investor's
goal is to minimize her worst-case final-time shortfall (i.e. ``regret'')
with respect to the (retrospectively) best-performing expert. Since the
focus is on worst-case behavior, this is a two-person game, whose players
are the investor and an adversarial market (which chooses the price
history so as to maximize the investor's final-time regret). To be clear:
the final time $N$ and the experts' investment rules $q$ and $r$ are
public information. At the $k$th timestep, the investor announces her
investment $f_k$, then the market chooses whether the stock goes up
($b_k=1$) or down ($b_k = -1$); these determine the investor's gain $b_k f_k$
and also the experts' gains $b_k q(m_k)$ and $b_k r(m_k)$; here $m_k$
represents $k$ days of history prior to timestep $k$. (For a more
complete description see Section \ref{sec2-1}.)

The investor's worst-case final-time regret can in principle be determined
by dynamic programming, using as state variables the investor's regret with
respect to each expert and the past $d$ stock price moves (see Section
\ref{sec2-3}). But one can do better by taking advantage of the fact that
regret accumulates slowly while the recent history changes at each time
step. Indeed, as time proceeds the sequence of recent histories determines
a walk on a certain directed graph (see Section \ref{sec2-2}). Simplifying
somewhat, let us suppose that the market's choices involve well-selected
cycles on this graph (a key task in our analysis of the investor's strategy
will be to justify this simplification). Then the rapidly-changing
recent history can be accounted for by considering, for each cycle, how regret
accumulates if the market chooses that cycle. This gives an estimate of the
worst-case final-time regret that's independent of the recent history -- a
function only of the investor's regret with respect to each expert and
number of timesteps that remain. (The case when only $d=1$ timestep of
history is used is especially transparent; it is discussed in detail
in Section \ref{bounds-d=1}.)

This game is discrete (in much the way that random walk on a lattice is
discrete). To connect the game with a PDE we consider a scaling limit
(entirely analogous to the way that scaled random walks lead to Brownian
motion, whose backward Kolmogorov equation is the linear heat equation).
To explain heuristically, suppose $v(j,y_1, y_2)$ is the worst-case regret
at the final time $N$, given that the investor's regret with respect to
experts $1$ and $2$ are $y_1$ and $y_2$ at timestep $j$ (and ignoring for
now the importance of the recent history). As already noted above, we
expect $v$ to be of order $\sqrt{N}$ when $(y_1,y_2)= (0, 0)$ and $j=0$,
by arguments related to Blackwell's theorem. Therefore it is natural to
consider the rescaled regret $u=v/\sqrt{N}$ as a function of $x=y/\sqrt{N}$
and $t=j/N$. Writing $\varepsilon = 1/\sqrt{N}$, the rescaled dynamic
programming principle for $u$ involves spatial steps of order $\varepsilon$
and timesteps of order $\varepsilon^2$. The relevant PDE for $u(x,t)$ is,
roughly speaking, the one for which this rescaled dynamic programming
principle is a convergent numerical scheme in the limit
$\varepsilon \rightarrow 0$.

The preceding discussion of scaling neglects the importance of the
recent stock price history. In fact, we identify a limiting PDE for the
investor's worst-case regret only when the experts use up to $d=4$
days of history. This is because we achieve a full understanding of
how the stock price history determines the players' optimal choices
only when $d \leq 4$. But we also have results for $d>4$, involving
upper (respectively lower) bounds for the worst-case regret, obtained by
solving PDEs associated with specific strategies for the investor
(respectively the market).

The PDE's that emerge are 2nd-order nonlinear parabolic equations, solved
backward in time. Such PDE are commonly seen in stochastic control; here
the second-order character comes not from stochasticity, but rather from
the game's Hannan consistency (a situation analogous to that of \cite{KS1}).
Based on our account of the game, the final-time value should be
$\varphi(x_1,x_2) = \max\{x_1,x_2\}$, since the predictor's goal was
to minimize her worst-case regret with respect to the best-performing
expert, i.e. $\max\{y_1,y_2\}$. (Since this function is homogeneous of
degree one, the passage from regret to scaled regret does not change its
form.) It is, however, both natural and convenient to consider the
scaled dynamic programming principle with a more general final-time
value $\varphi(x)$. Indeed, our methods rely on the smoothness of $u$,
so they require $\varphi$ to be sufficiently smooth. We adapt them to
the classic case $\varphi(x) = \max\{x_1,x_2\}$ by approximating it
(above or below) by a smooth function.

What are the PDE's? Actually, there is essentially just \emph{one} PDE.
It is convenient to change variables to $\xi = x_1-x_2$ and $\eta = x_1 + x_2$.
The PDE is then
\begin{equation} \label{pde-intro}
u_t + C u_\eta^{-2}
\left( u_{\xi \xi} u_\eta^2 - 2 u_{\xi \eta} u_\xi u_\eta + u_{\eta \eta} u_\xi^2 \right)
 = 0 \quad \mbox{for $t < T$}
\end{equation}
with the final-time condition $u(T, \xi, \eta) = \varphi(\xi,\eta)$. The
constant $C$ must of course be chosen properly; it reflects the rate at
which regret accumulates (see Section \ref{sec2-3} for an account of how
such a PDE emerges from the dynamic programming principle). The PDE
(\ref{pde-intro}) looks very nonlinear, but it simplifies dramatically
when the final-time condition has the form
$\varphi(\xi,\eta) = c \eta + \overline{\varphi}(\xi)$. (Note that the
classic choice $\max(x_1,x_2)$ has this form, since it equals
$\frac{1}{2} (\eta + |\xi|)$.) Then it is natural to expect that
$u(t,\xi,\eta) = c \eta + \overline{u}(t,\xi)$, and substitution
of this ansatz into the PDE gives the linear heat equation
$\overline{u}_t + C \overline{u}_{\xi \xi} = 0$ with the final-time
condition $\overline{u}=\overline{\varphi}$ at $t=T$. Surprisingly,
something similar can be done for a much broader class of final-time
conditions $\varphi(\xi, \eta)$. Indeed, assuming certain structural
conditions on $\varphi$, for each $y$ the level set $u(t,\xi, \eta) = y$
has the form $\eta =  h(t,\xi; y)$ where $h$ solves a linear heat
equation in $\xi$ and $t$, with appropriate final-time data. This is
explained in \cite{Zhu}; we offer a self-contained treatment as an
Appendix.

How is the investor's strategy linked to the PDE? For two
static experts (i.e. when the two experts make no use of the recent
stock price history) the investor's optimal strategy resembles what
is known in the literature as a potential-based strategy, using the
solution of the PDE as a (time-dependent) potential \cite{Zhu}. In
our history dependent setting, the situation is quite different: the
investor's optimal strategies depend on the recent stock price history
as well as on the solution of the PDE. Our analysis suggests that if
the recent history is $m$, the investor should buy
$f_m = f_m^* + \varepsilon f_m^\#$ units of stock. The
leading-order term $f_m^*$ is state-dependent, but it is
nevertheless obtained by using the solution of the PDE as
something like a potential; the correction $\varepsilon f_m^\#$
is obtained by an entirely different argument, involving the graph
that represents the evolution of stock price histories (see
Sections \ref{sec2-4}, \ref{sec3-1}, and \ref{sec4-1} for further
information in this direction).

What about the market's strategies? Briefly (and informally),
knowing that the investor should choose a strategy of the form
$f_m = f_m^* + \varepsilon f_m^\#$, the market identifies
a particular choice of $f_m^\#$ that works to its advantage,
and ``forces'' the investor to use this strategy by penalizing
other choices (see Sections \ref{sec2-4}, \ref{sec3-2} and \ref{sec4-2} for further
information in this direction).

To put our work in context with respect to the online machine learning
literature, we note that there are numerous articles on the ``prediction
of binary sequences,'' of which \cite{CBL2} is representative. In much of
this literature, the focus is on identifying easily-implemented strategies
for the predictor  (involving ``potential functions'' for example) that
work relatively well (assuring regret of order at most $\sqrt{N}$ after
$N$ timesteps). Our focus is different: we would like to understand an
asymptotically {\it optimal} strategy for the predictor, and the optimal
prefactor of $\sqrt{N}$ in the estimate of the worst-case regret. This
requires identifying asymptotically optimal strategies for the market
as well as for the predictor. Such a complete understanding has to
date been achieved for only a few prediction problems. In the present
paper we achieve such clarity when the history-dependence
involves at most $d =4$ recent stock movements. For $d > 4$, as noted above, we
obtain upper and lower bounds for the worst-case regret, but we
don't know whether they match; this question reduces, as we
explain in Section \ref{LP}, to one about the cycles
in certain de Bruijn graphs.

The online machine learning literature is not restricted to the
prediction of sequences. A different class of examples focuses
directly on the experts' outcomes, letting those be determined
directly (rather than through a time series) by the market. (A
PDE-based discussion of one such problem can be found in \cite{DK},
and a PDE perspective on potential-based strategies can be
found in \cite{KKW1,KKW2,Rokhlin}; PDE methods have also been applied
to another class of problems known as ``drifting games'' \cite{FO}.)
In the present setting the experts' outcomes are highly constrained,
since (i) their progress must be consistent with a time series, and
(ii) their predictions depend, at a given time, on the past $d$ items
in that time series. The graph-theoretic methods we use to deal with
the fast time scale are essentially a means for understanding the effect
of these constraints. Perhaps related methods might be useful for some
other prediction problems with similar constraints on the experts' advice.

In a generic problem involving prediction of sequences, the goal might
be to actually guess the next element of the sequence. In such a
setting the ``loss function'' would measure the error; a typical
example would be the ``absolute value loss'' $|f-b_k|$. In the
stock prediction problem the situation is different -- the investor
and the experts focus instead on how much money they gain or lose.
This seems at first glance very different from the absolute value loss;
however in the context of the stock prediction problem (where each
stock move $b_k$ is limited to $\pm 1$, and the focus is exclusively
on cumulative regret) the use of absolute value loss rather than
financial loss would not really change matters. (This fact is well-known and
elementary; for completeness we nevertheless review it in
Section \ref{sec2-1}.)

We briefly highlight a few recent papers with connections to our work.
The stock prediction problem with two static experts, a discounted
payoff, and no final time was studied in \cite{AP}. The focus there
(like here) is on identification of the investor's and the market's
optimal behavior. While the paper's methods seem somewhat different
from ours, differential equations do play a fundamental role. A
related though somewhat different discussion can be found in
Section 1.5.2 of \cite{Zhu}, where the problem is considered
for a discounted payoff and a fixed final time.

The stock prediction problem with two static experts is a special
case of a more general class of problems studied in \cite{KP}. The
focus there is rather different from ours: rather than considering
just the predictor's regret (i.e. her shortfall with respect to the
best-performing expert), the paper \cite{KP} discusses algorithms
that limit the worst-case loss as well as the worst-case regret.
The analysis there permits any number of experts.

The rest of this paper is organized as follows: Section \ref{GS}
sets the stage, by presenting our problem in full
detail and linking it heuristically to our PDE; it closes, in Section
\ref{sec2-6}, with a summary of our main results. Section \ref{bounds-d=1}
proves our upper and lower bounds in the special case when the experts use
only $d=1$ days of history; this case is treated separately because
it is relatively transparent, while still capturing the main ideas
used for general $d$. Section \ref{LP} introduces the graph-based
methods used to determine the investor's and market's strategies.
A good strategy for the investor leads to an upper
bound for the worst-case final-time regret, while a good strategy
for the market leads to a lower bound for the same quantity; our
proofs of these bounds occupies Section \ref{bounds-d>1}.
Our analysis relies on the PDE (\ref{pde-intro}) having a sufficiently smooth
solution (with uniformly bounded derivatives) when the
final-time function $\varphi$ satisfies some structural properties.
The required estimates follow easily from basic facts about the linear
heat equation when the final-time function has the form
$\varphi(\xi,\eta) = c \eta + \overline{\varphi}(\xi)$. They are, however,
valid for a broader class of final-time functions $\varphi$;
this is proved in \cite{Zhu} but we provide a self-contained
treatment here as an Appendix, for the convenience of the reader.

\section{Getting started} \label{GS}

This section describes our problem in detail and provides a heuristic discussion of
our approach. It closes, in Section \ref{sec2-6}, with a summary of our main results.

\subsection{The game} \label{sec2-1}

The stock prediction problem is a zero-sum, two-person, sequential game played by the
investor and the market. We represent the stock movements by a binary data stream
$b_1, b_2, \ldots$, each with value $-1$ or $+1$.

Since our experts make their predictions using $d$ days of history, we
need some notation for the most recent $d$ stock movements (the current ``state'').
While the obvious representation is $(b_{k-d},\ldots,b_{k-1})$ at time $k$, it is
convenient to use binary notation instead, recording $-1$ as $0$ and $+1$ as $1$. With this
convention, the state at time $k$ is
$$
m_k = \frac{1}{2}(b_{k-d}+1,\dots, b_{k-1}+1) \in  \{0,1\}^d,
$$
which can be viewed as an integer ranging from $0$ to ${2^d}-1$.

At each time step the investor can buy $|f| \leq 1$ units of stock. Her choice of $f$ can
be viewed as a prediction, since if she buys $f_k$ shares at time $k$ and the stock
then moves by $b_k$, she has a profit of $b_k f_k$ (or a loss, if this is negative).
Our two history-dependent experts' predictions are expressed similarly: the experts are
characterized by two functions $q(m)$ and $r(m)$, defined on $\{0,1\}^d$, which represent
their investment choices. The experts should of course not be identical,
\begin{equation} \label{q-and-r-distinct}
\mbox{$q(m) \neq r(m)$ for at least one state $m$},
\end{equation}
and their guidance should be admissible: $|q(m)| \leq 1$ and $|r(m)| \leq 1$
for all $m$. For technical reasons, we need the preceding inequalities to be strict:
\begin{equation} \label{cond-on-experts}
\mbox{we assume that $|q(m)|< 1$ and $|r(m)| < 1$ for all states $m$.}
\end{equation}
(This condition assures that the investor's leading-order
bid $f_m^*$ satisfies $|f_m^*| < 1$, as we'll see in Section \ref{sec2-5}.)

Following the convention of the prediction literature, we focus on the investor's and
the experts' losses rather than their gains. In terms of the \emph{financial loss} function
\begin{equation} \label{financial-loss}
l(f,b_k) = -fb_k
\end{equation}
the investor's and experts' losses due to stock movement $k$ are
evidently $l(f_k,b_k)$, $l(q_k, b_k)$, and $l(r_k, b_k)$, if the investor's prediction
is $f$, the recent history is $m_k$, and the experts' bids are $q_k=q(m_k)$ and
$r_k = r(m_k)$. So the investor's \emph{regret} (or shortfall, or relative loss)
with respect to expert $q$ at step $k$ is
\begin{equation} \label{one-step-regret-q}
l(f, b_k)-l(q_k, b_k) = -fb_k + q_k b_k =b_k(q_k-f),
\end{equation}
and her regret with respect to expert $r$ at step $k$ is similarly
\begin{equation} \label{one-step-regret-r}
l(f, b_k)-l(r_k, b_k) =-fb_k +r_k b_k=b_k(r_k-f) .
\end{equation}
The investor's cumulative regret is a key quantity,
since it captures her overall performance compared to the experts.
We record the cumulative regret with respect to experts $q$ and $r$ by $(x_1,x_2)$.
To be explicit: if the investor's prediction at step $k$ is $f_k$ then after the $M$th
step,
\begin{equation} \label{cumulative-regret}
x_1 = \sum_{k=1}^M l(f_k, b_k)-l(q_k, b_k) \quad \mbox{and} \quad
x_2 = \sum_{k=1}^M l(f_k, b_k)-l(r_k, b_k).
\end{equation}

The language of investment makes the financial loss function $l(f,b) = -fb$ seem very natural.
But (as noted in the Introduction) the situation would be no different if we used the absolute
value loss function $l'(f,b) = |f-b|$. Indeed, since $b$ takes only the values $\pm 1$, the associated
regret with respect to expert $q$ at time $k$ would be
$$
l'(f, b_k)-l'(q_k, b_k) = |f-b_k| - |q_k -b_k| =b_k(q_k-f)
$$
and similarly the associated regret with respect to expert $r$ would be $b_k(r_k-f)$. Comparing
with \eqref{one-step-regret-q}--\eqref{one-step-regret-r}, we see that when only the regret (rather
than the total loss) is being considered, the financial loss and absolute value loss
functions are equivalent.

When the game starts there is no history; we suppose the two experts have some rule for handling this (for example
$q_k = r_k = 0$ if $1 \leq k \leq d$). The specific choice of this rule will have no effect on our analysis, since
we are interested in what happens over many time steps.

\subsection{The underlying graph} \label{sec2-2}

Since the state $m \in \{0,1\}^d$ records the recent history, it changes from one step to the next.
For example, if $d=2$ and the most recent moves were $b_{k-2}=-1$ and $b_{k-1}=1$ then $m_k = (0,1)$
and the next state $m_{k+1}$ can be either $(1,1)$ (if $b_k = 1$) or $(1,0)$ (if $b_k = -1)$. Introducing
some notation: if the current state is $m$, we define
\begin{eqnarray*}
m_{+} & = & \mbox{subsequent state if $b_k=1$, and}\\
m_{-} & = & \mbox{subsequent state if $b_k = -1$}.
\end{eqnarray*}
Thus when
$$
m=\frac{1}{2}(b_{k-d}+1,\dots, b_{k-1}+1) \in  \{0,1\}^d
$$
we have
\begin{eqnarray*}
m_{+} & = & \frac{1}{2}(b_{k-d+1}+1,..., b_{k-1}+1,2), \mbox{ and}\\
m_{-} & = & \frac{1}{2}(b_{k-d+1}+1,..., b_{k-1}+1,0).
\end{eqnarray*}
As the game proceeds, the sequence of states can be visualized using an appropriate directed graph.
It has the states as vertices, and each vertex $m$ is connected to $m_{+}$ and $m_{-}$; the case $d=2$ is shown
in Figure \ref{fig:graph(d=2)withoutcosts}. Evidently the graph has $2^d$ vertices; each vertex has exactly
two outgoing edges; and each vertex has exactly two incoming edges. It is, in fact, a well-studied graph, known
as the $d$-dimensional de Bruijn graph on $2$ symbols (see e.g. page 61 of \cite{West}).
\begin{figure}[h]
\centerline{\includegraphics[width=3in]{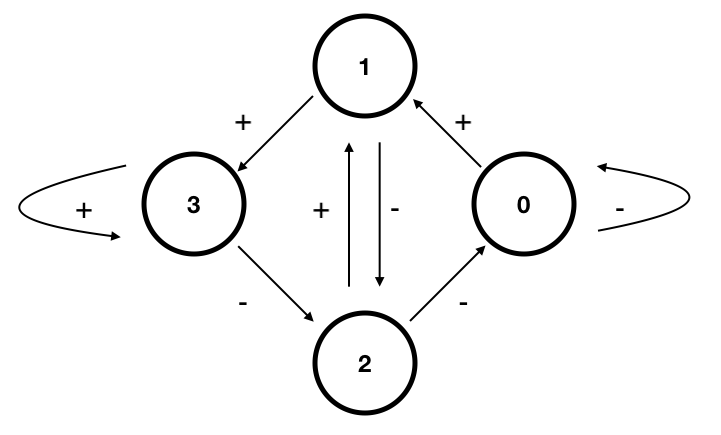}}
\caption{the underlying graph, when $d=2$.}
\label{fig:graph(d=2)withoutcosts}
\end{figure}

Our analysis uses the well-known fact that a closed walk can be decomposed as a union of
simple cycles. To keep our treatment self-contained, we review the relevant
definitions and give the proof of this result.

\begin{definition}
A closed walk is a list of vertices and edges $a_0, e_1, a_1,\dots, e_k, a_k$ such that the initial and final vertices are
the same ($a_k=a_0$), and for $1\leq i \leq k$ the edge $e_i$ has endpoints $a_{i-1}$ and $a_i$.
\end{definition}

\begin{definition}
A simple cycle is a walk without repeated vertices apart from the first one and the last one, which are equal.
\end{definition}

\noindent Where it introduces no confusion, we will identify a cycle by listing its vertices in order. As an example,
when $d=2$ (the case shown in Figure \ref{fig:graph(d=2)withoutcosts}) there are $6$ simple
cycles: two with length one ($00$ and $33$), one with length two ($121$), two with length three
($0120$ and $1321$), and one with length four ($01321$).

\begin{lemma} \label{walkcycles}
A closed walk on a directed graph is a union of simple cycles.
\end{lemma}
\begin{proof} The argument is constructive. Given a closed walk, consider the list of
vertices it visits (in order): $a_0, a_1, a_2, \dots, a_k$. (By hypothesis, $a_0=a_k$).
Now start traversing the walk, and consider the first time a vertex is repeated: $a_i=a_j$ with $j > i$.
Evidently $a_i, a_{i+1}, \dots, a_j$ is a simple cycle. Remove this simple cycle from the walk, keeping
$a_i$. This gives a new (shorter) closed walk, to which the same argument can be applied. Continuing,
we eventually get a walk that has no repetitions aside from its intial and final vertices. It is evidently
a simple cycle, so the proof is complete.
\end{proof}

\subsection{The dynamic programming principle before and after scaling} \label{sec2-3}

The investor's worst-case regret is determined by a dynamic programming principle. The goal of this
subsection is to make this explicit. As already explained in the introduction, our main focus will be
on a \emph{scaled} version of the game, and we'll get to that presently. But first we discuss the
dynamic programming problem for the game described by Section \ref{sec2-2}.

Recall that the investor's cumulative regret is recorded by $(x_1,x_2)$. As already noted in
the Introduction, it is convenient to work with $\xi = x_1-x_2$ and $\eta = x_1 + x_2$. If
the investor's choice at step $k$ is $f_k$ then the change in $x$ associated with actions
at step $k$ is, according to (\ref{one-step-regret-q})--(\ref{cumulative-regret}),
$\Delta x = [b_k(q_k-f_k),b_k(r_k-f_k)]$. So the changes in $\xi$ and $\eta$ are
$\Delta \xi = b_k (q_k - r_k)$ and $\Delta \eta = b_k (q_k + r_k - 2f_k)$.
We suppose the game ends at step $N$ (i.e. the last prediction is made at $k=N-1$).

The dynamic programming principle determines the function
$$
U(k,m,\xi,\eta) = \left\{ \begin{tabular}{l}
\mbox{the investor's worst-case minimum final-time regret at step $k$, if the recent history is $m$,}\\
\mbox{the value of $x_1-x_2$ (based on times through $k-1$) is $\xi$, and the value of $x_1 + x_2$ is $\eta$,}
\end{tabular}
\right.
$$
by working backward in time: for $k \leq N-1$
\begin{equation} \label{dpp-discrete}
U (k, m, \xi, \eta) = \min_{|f|\leq 1} \max_{b_k=\pm 1}
U \bigl( k+1, m_{b_k}, \xi +b_k [q(m)-r(m)], \eta + b_k[q(m)+r(m)-2f] \bigr),
\end{equation}
where $m_{b_k}$ denotes the next state, which is determined from $m$ by the value of $b_k$.
Since the investor's regret with respect to the best-performing expert is
$\max\{x_1,x_2 \} = (\eta + |\xi|)/2$, the final-time condition is
\begin{equation} \label{final-time-discrete}
U(N,m,\xi,\eta) = \frac{\eta + |\xi|}{2}.
\end{equation}
Our hypotheses about the flow of information are reflected by the min-max in
\eqref{dpp-discrete}: the investor chooses $f$ knowing that the adversarial market
will see her choice and do whatever is worst for her.

We are interested in what happens over large numbers of steps. As discussed in the
Introduction, we expect $U(1,m,0,0) \sim \sqrt{N}$. Therefore it is natural to
scale time by $N$ and regret by $\sqrt{N}$. Writing $\varepsilon = 1/\sqrt{N}$, the
scaled version of $U$ is
$$
u^\varepsilon (t,m,\xi,\eta) =
\varepsilon U \left( \frac{t}{\varepsilon^2}, m, \frac{\xi}{\varepsilon}, \frac{\eta}{\varepsilon} \right).
$$
It is defined for times $t$ that are multiples of $\varepsilon^2 = 1/N$ between $0$ and $1$. The dynamic
programming principle (\ref{dpp-discrete}) is equivalent to
\begin{equation} \label{DPP}
u^\varepsilon(t, m, \xi, \eta) = \min_{|f|\leq 1} \max_{b_t=\pm 1}
u^\varepsilon \bigl( t+\varepsilon^2, m_{b_t}, \xi +\varepsilon b_t[q(m)-r(m)], \eta +\varepsilon b_t[q(m)+r(m)-2f] \bigr)
\end{equation}
with the same final-time condition as before:
\begin{equation} \label{FT-classic}
u^\varepsilon(1, m, \xi, \eta)=  \frac{\eta + |\xi|}{2}.
\end{equation}

Our goal is to study $u = \lim_{\varepsilon \rightarrow 0} u^\varepsilon$ by combining ideas from PDE and
graph theory. In our study of the continuum limit, it is natural to consider any final time $T$ (rather than
just $T=1$), and it is natural to consider other final-time conditions (not just $(\eta + |\xi|)/2$). Therefore
we shall henceforth assume that $u^\varepsilon$ satisfies (\ref{DPP}) up to a fixed final time $T$, with
a more general final-time condition
\begin{equation} \label{FT-general}
u^\varepsilon(T, m, \xi, \eta)=  \varphi(\xi,\eta).
\end{equation}
Our analysis requires some conditions on the final-time condition $\varphi$;
these will be discussed in Section \ref{sec2-5}.

\subsection{A formal connection to PDE} \label{sec2-4}

To see a connection with PDE, we begin by ignoring the dependence of
$u^\varepsilon$ on $\varepsilon$ and $m$, and combining the dynamic programming principle
(\ref{DPP}) with Taylor expansion. Replacing $u^\varepsilon$
by $u$ in \eqref{DPP} and assuming $u$ is sufficiently smooth, this gives
\begin{align} \label{TaylorMinMax}
u(t,m,\xi,\eta) &= \min_{|f_{t,m}|\leq 1} \max_{b_t=\pm 1}
u\bigl( t+\varepsilon^2, m_{b_t}, \xi +\varepsilon b_t[q(m)-r(m)], \eta +\varepsilon b_t[q(m)+r(m)-2f] \bigr) \nonumber \\
&= \min_{|f_{t,m}|\leq 1} \max_{b_t=\pm 1}  \left\{ u(t,m,\xi,\eta) + \varepsilon A + \varepsilon^2 B +
O(\varepsilon^3) \right\}
\end{align}
with
\begin{equation} \label{defn-of-A}
A=b_t \bigl( [q(m)-r(m)]u_{\xi}+[q(m)+r(m)-2f_{t,m}]u_{\eta} \bigr)
\end{equation}
and
\begin{equation} \label{defn-of-B}
B = u_t + \frac{\langle D^2u [q(m)-r(m),q(m)+r(m)-2f_{t,m}], [q(m)-r(m),q(m)+r(m)-2f_{t,m}] \rangle}{2}.
\end{equation}
In the latter expression, $D^2 u$ is the Hessian of $u$ with respect to its spatial variables $\xi$ and $\eta$:
\begin{equation}
D^2 u = \begin{pmatrix}
 u_{\xi\xi}& u_{\xi\eta} \\
 u_{\xi\eta}& u_{\eta\eta}
 \end{pmatrix} .
 \end{equation}
The order $\varepsilon^0$ term in the min-max (\ref{TaylorMinMax}) is not interesting -- it
is independent of $f_{t,m}$ and $b_t$, and in fact it cancels
with the $u$ that's on the left. So the first interesting term in the min-max is $\varepsilon A$.
If it is nonzero, then the market can choose the sign of $b_t$ to make this term positive -- a bad outcome
for the investor. It is tempting to conclude that the investor's choice of $f_{t,m}$ should make $A$
vanish, but that's not quite right. Rather, the investor's choice of $f_{t,m}$ should
make this term small enough that it interacts with the $\varepsilon^2$ term.  This occurs when
$A = O(\varepsilon)$, i.e.
$$
[q(m)-r(m)]u_{\xi}+[q(m)+r(m)-2f_{t,m}]u_{\eta} = O(\varepsilon),
$$
or equivalently
$$
f_{t,m}=\frac{[q(m)-r(m)]u_{\xi}+[q(m)+r(m)]u_{\eta}}{2u_{\eta}} + O(\varepsilon).
$$
So the investor's choice should take the form
\begin{equation} \label{f_m_old}
f_{t,m}=f^*_{t,m}+\varepsilon f^{\#}_{t,m}
\end{equation}
with
\begin{equation} \label{f^*m}
f^*_{t,m}=\frac{[q(m)-r(m)]u_{\xi}+[q(m)+r(m)]u_{\eta}}{2u_{\eta}}.
\end{equation}
(We shall explain later, in Section \ref{sec2-5}, why this choice is admissible -- that is, why
$|f_{t,m}| \leq 1$ -- under suitable hypotheses upon the final-time function $\varphi$.)
When $f_{t,m}$ has the form (\ref{f_m_old}), the ``first-order term'' from the Taylor expansion becomes
\begin{align*}
\varepsilon A & =
\varepsilon b_t \bigl( [q(m)-r(m)]u_{\xi}+[q(m)+r(m)-2(f^*_{t,m}+\varepsilon f^\#_{t,m})]u_{\eta} \bigr) \\
&=-\varepsilon^2 2 b_t u_{\eta}f^\#_{t,m}.
\end{align*}
As for the ``second-order term'' $\varepsilon^2 B$: evaluating $B$ at $f_{t,m}^*$ rather than
$f_{t,m}$ introduces an error of order $\varepsilon$; this leads (after some algebra) to
\begin{equation} \label{value-of-B}
\varepsilon^2 B = \varepsilon^2 \left( u_t +(q(m)-r(m))^2
\frac{1}{2}\langle D^2u \frac{\nabla^\bot u}{u_{\eta}},\frac{\nabla^\bot u}{u_{\eta}}\rangle  \right)
+ O(\varepsilon^3),
\end{equation}
using the notation
\begin{equation} \label{uperp}
\nabla^\bot u = (-u_{\eta},u_{\xi}) .
\end{equation}
Summarizing the preceding calculation: when $f$ has the form (\ref{f_m_old}), \eqref{TaylorMinMax} reduces (after division by $\varepsilon^2$) to
\begin{equation} \label{TaylorMinMax-reduced}
0 = \min_{f_{t,m}^\#} \max_{b_t=\pm 1}
\left\{ u_t +(q(m)-r(m))^2
\frac{1}{2}\langle D^2u \cdot \frac{\nabla^\bot u}{u_{\eta}},\frac{\nabla^\bot u}{u_{\eta}}\rangle -
2 b_t u_{\eta}f^\#_{t,m} \right\}
\end{equation}
if we ignore the error associated with higher-order terms in the Taylor expansion.

We have thus far ignored the possible dependence of $u$ on $m$. It seems from \eqref{TaylorMinMax-reduced} that
$u$ should indeed depend on the state $m$ as well as on $t$, $\xi,$ and $\eta$. This would be a mess, since we would
need to keep track of how the states evolve.

In fact, the function $u$ we want is a little different from the one discussed thus far. The
min-max should be over the players' (multistep) \emph{strategies}, and the expression under the min-max
should involve the \emph{time-averaged value} of the expression under the min-max in
\eqref{TaylorMinMax-reduced}. By proceeding this way, we shall obtain a function $u$ that depends only on time
and space (not on the current state).

To avoid unwieldy expressions, we introduce the notation
\begin{equation} \label{D_k}
D_k :=
\frac{1}{2}\langle D^2u \cdot \frac{\nabla^\bot u}{u_{\eta}}, \frac{\nabla^\bot u}{u_{\eta}}\rangle,
\end{equation}
where the right hand side is evaluated at time $t_k$ and location $\xi_k, \eta_k$. The expression to be time-averaged
is thus
\begin{equation} \label{defn-of-L}
L(t_k, m, b, \xi_k, \eta_k, f^\#_{t_k, m}):=u_t +(q(m)-r(m))^2 D_k-2b u_{\eta} f^\#_{t_k, m}.
\end{equation}

To identify the optimal strategies and implement the time averaging, let us focus on the case $d=1$,
for which the graph is shown in Figure \ref{fig:graph(d=1)withoutcosts}. It has three simple
cycles: $0-0$, $1-1$, and $0-1-0$.
\begin{figure}[h]
\centerline{\includegraphics[width=3in]{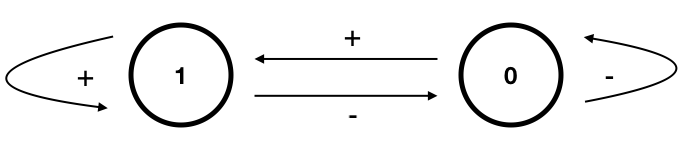}}
\caption{The directed graph when $d=1$.}
\label{fig:graph(d=1)withoutcosts}
\end{figure}
The investor knows that the market could choose to simply repeat one of these cycles.
For the cycle $0-0$, the value of $L$ is
\begin{equation} \label{cycle00}
L(t_k, 0, -1, \xi_k, \eta_k, f^\#_{t_k, 0}) = u_t +(q(0)-r(0))^2 D_k-2(-1)u_{\eta} f^\#_{t_k, 0};
\end{equation}
for the cycle $1-1$ it is
\begin{equation} \label{cycle11}
L(t_k, 1, 1, \xi_k, \eta_k, f^\#_{t_k, 1})=u_t +(q(1)-r(1))^2 D_k-2(+1)u_{\eta} f^\#_{t_k, 1};
\end{equation}
for the two-step cycle $0-1-0$ the time-averaged value is
\begin{multline} \label{cycle010}
\frac{1}{2} \{ L(t_k, 1, -1, \xi_k, \eta_k, f^\#_{t_k, 1})+ L(t_k, 0, 1, \xi_k, \eta_k, f^\#_{t_k, 0}) \} \\
= u_t + \frac{1}{2} \{ (q(1)-r(1))^2 D_k-2(-1)u_{\eta} f^\#_{t_k, 1} + (q(0)-r(0))^2 D_k-2(+1)u_{\eta} f^\#_{t_k, 0} \}.
\nonumber
\end{multline}
Knowing that the market is adversarial, the investor assumes that the market will choose whichever cycle gives her
the worst result, and chooses $f^\#_{t,m}$ to optimize the outcome associated with the worst-case cycle. This leads, for
general $d$, to a linear program (the ``investor's linear program'' studied in Section \ref{sec4-1}). For $d=1$, however, the
solution of the linear program has a simple and intuitive character: the best choice of $f^\#_{t,m}$ is the one that makes
the investor \emph{indifferent} with respect to the three simple cycles. (A similar phenomenon occurs for $d=2,3,4$, as
we will show in Section \ref{sec4-4}.) Writing $D =D_k$ (a harmless convention, since $D_k$ changes only slightly at each
timestep) and writing $\alpha_m$ rather than $f^\#_{t,m}$ (a significant choice, since it restricts the investor's strategy
to depend on time only through the state $m$ and the values of $u_\eta$ and $D$), the condition for indifference is that
$\alpha_0$ and $\alpha_1$ satisfy
\begin{align*}
(q(1) - r(1))^2 D - 2 u_\eta \alpha_1 & = (q(0)-r(0))^2 D + 2 u_\eta \alpha_0 \\
& = \frac{1}{2} \left[(q(0)-r(0))^2 + (q(1) - r(1))^2 \right] D + u_\eta \alpha_1 - u_\eta \alpha_0 .
\end{align*}
This amounts to two linear equations in the two unknowns $\alpha_0$ and $\alpha_1$; the unique solution is easily seen to be
$$
\alpha_0 = \alpha_1 = \frac{(q(1)-r(1))^2 - (q(0) - r(0))^2}{4 u_\eta} D
$$
and the time-averaged value of $L$ for each cycle is then
\begin{equation} \label{time-averaged-value}
u_t + \frac{(q(0)-r(0))^2 + (q(1) - r(1))^2}{2} D.
\end{equation}
Taking this choice of strategy into account, we obtain from \eqref{TaylorMinMax-reduced}
the desired PDE for the case $d=1$:
\begin{equation} \label{PDEsharp1a}
u_t + \frac{1}{2} C^\#_1 \langle D^2 u \frac{\nabla^\bot u}{ u_{\eta}}, \frac{\nabla^\bot u}{ u_{\eta}} \rangle = 0
\end{equation}
with
\begin{equation} \label{PDEsharp1b}
C^\#_1 = \frac{(q(0)-r(0))^2 + (q(1) - r(1))^2}{2}.
\end{equation}
It is to be solved for $t<T$, with final-time data $u=\varphi$ at time $T$.

The preceding derivation is of course quite heuristic. We will justify it with full mathematical rigor in
Section \ref{bounds-d=1}. The ideas presented above provide a strategy for the investor, and the
outline of a proof that by following it she can do at least as well as the solution of the PDE. The heuristic argument
considered only simple cycles, while the rigorous one (presented in Section \ref{sec3-1}) must consider any possible
actions the market might take. This is where Lemma \ref{walkcycles} comes in: while the market's choices might
not produce a walk that's restricted to a simple cycle, the walk they produce can be decomposed as a union of
simple cycles -- so a strategy that works well for every simple cycle actually works well for every walk.

How do we know there is no \emph{better} strategy for the investor? This will be shown in Section \ref{sec3-2} by
considering a particular strategy for the market. Briefly, the strategy ``forces'' the investor to use the values of
$f_m^*$ and $f_m^\#$ identified by our heuristic argument, by penalizing other choices.

Our heuristic discussion has focused on the case $d=1$. The cases $d=2,3,4$ are similar, since in those cases too
the parameters $f_m^\#$ can be chosen to make the investor indifferent with respect to the various simple
cycles; the associated
PDEs have the same form as \eqref{PDEsharp1a}, except that the ``diffusion constant'' is different. (The values of the
constants $C_2^\#$, $C_3^\#$, and $C_4^\#$, corresponding to $d=2,3,4$, are determined in Section \ref{sec4-4}.)
For $d \geq 5$
we do not know whether indifference is achievable. As a result, for $d\geq 5$ our method only gives upper and lower
bounds for the investor's worst-case final-time regret. Both bounds involve PDE's of the form
\eqref{PDEsharp1a} with suitable ``diffusion constants.'' The arguments are largely
parallel to those sketched in this section for $d=1$. To find a good strategy for the investor, we
choose the parameters $f_m^\#$ to minimize the rate at which regret accumulates for the most dangerous cycle. To find
a good strategy for the market, we choose the parameters $f_m^\#$ to maximize the rate at which regret accumulates
for the cycle most favorable to the investor. The optimal choices of $f_m^\#$ are determined by a pair of linear programs,
presented in Section \ref{LP}. One might have expected these linear programs to be dual, but they are not;
Section \ref{sec4-3} explains why not.

\subsection{Properties of $u$ required for our analysis} \label{sec2-5}

The heuristic arguments in Section \ref{sec2-4} rely heavily on Taylor expansion, and so do our proofs.
Therefore we need that our PDE has a solution $u(t,\xi,\eta)$ that is sufficiently smooth, in the
sense that
\begin{equation} \label{require-u-smooth}
\begin{gathered}
\mbox{$u$ has continuous, uniformly bounded derivatives in $(\xi, \eta)$ of order up to $4$,}\\
\mbox{$u_t$ has continuous, uniformly bounded derivatives in $(\xi, \eta)$ of order up to $2$, and}\\
\mbox{$u_{tt}$ is continuous and uniformly bounded,}
\end{gathered}
\end{equation}
all bounds being uniform as $t \uparrow T$.
Our setup also has some additional requirements: it requires that
\begin{equation} \label{require-eta-deriv-pos}
u_\eta > c > 0
\end{equation}
(with a lower bound $c$ that is uniform in space and valid for all times between $t$ and $T$), and
\begin{equation} \label{require-ineq-btwn-xi-and-eta-derivs}
|u_\xi| \leq u_\eta .
\end{equation}
In Section \ref{bounds-d>1} we shall also need
\begin{equation} \label{require-time-deriv-neg}
u_t \leq 0.
\end{equation}
The relevance of \eqref{require-eta-deriv-pos} is clear, since many of the expressions that emerged from our
heuristic discussion had $u_\eta$ in the denominator. To explain the relevance of
\eqref{require-ineq-btwn-xi-and-eta-derivs}, we recall that $\xi = x_1-x_2$ and
$\eta = x_1 + x_2$, where $(x_1,x_2)$ are the investor's regrets with respect to the
two experts. By chain rule, $u_{x_1} = u_\xi + u_\eta$ and $u_{x_2} = -u_\xi + u_\eta$; so
\eqref{require-ineq-btwn-xi-and-eta-derivs} is equivalent to $u_{x_1} \geq 0$ and $u_{x_2} \geq 0$.
These conditions are quite intuitive: since regret accumulates as the game proceeds,
starting at time $t$ with more regret should mean ending at time $T$ with more regret. They are
also necessary for our heuristic calculation to make sense. Indeed, recall from \eqref{f_m_old} that if the
investor is in state $m$ at time $t$, her choice takes the form $f_{t,m} = f_{t,m}^* + \varepsilon f_{t,m}^\#$,
where $f_{t,m}^*$ is given by \eqref{f^*m}. Rewriting the formula for $f_{t,m}^*$ in terms of $u_{x_1}$ and
$u_{x_2}$ instead of $u_\xi$ and $u_\eta$, one finds that
\begin{equation} \label{u_x1+u_x2}
f^*_{t,m}= q(m)\frac{u_{x_1}}{u_{x_1}+u_{x_2}}+r(m)\frac{u_{x_2}}{u_{x_1}+u_{x_2}}.
\end{equation}
Since $u_{x_1}$ and $u_{x_2}$ are nonnegative (and their sum is strictly positive, by \eqref{require-eta-deriv-pos}),
$f^*_{t,m}$ is a weighted average of the experts' choices $q(m)$ and $r(m)$. The rules of the game
require that $|f_{t,m}| \leq 1$. Since $f_{t,m} = f_{t,m}^* + \varepsilon f_{t,m}^\#$
and we don't know the sign of $f_{t,m}^\#$ in advance, we need $|f_{t,m}^*| < 1$. This is assured by
\eqref{u_x1+u_x2}, together with our hypothesis \eqref{cond-on-experts}) that
$|q(m)| < 1$ and $|r(m)| < 1$ for every state $m$.

Condition \eqref{require-time-deriv-neg} reflects again the idea that regret accumulates as the game proceeds.
If we start the game a little later (with the same initial regrets, and the same final time $T$) then
the total time that the game runs is smaller, so the final-time regret should be smaller. Thus
$u(t,\xi,\eta)$ should be a decreasing function of $t$ when $\xi$ and $\eta$ are held fixed.

These conditions on $u$ are effectively conditions on the final-time data $\varphi$. As already noted
in the Introduction, it is very natural to focus on final-time data of the form
\begin{equation} \label{special-final-time-data}
\varphi(\xi,\eta) = c\eta + \overline{\varphi}(\xi).
\end{equation}
Then the PDE
\begin{equation} \label{our-pde-repeated}
u_t + \frac{1}{2} C \langle D^2 u \frac{\nabla^\bot u}{ u_{\eta}}, \frac{\nabla^\bot u}{ u_{\eta}} \rangle = 0
\end{equation}
with $u = \varphi$ at $t=T$ is solved by $u = c\eta + \overline{u}(t,\xi)$, where
\begin{equation} \label{reduces-to-heat-equation}
\overline{u}_t + \frac{1}{2} C \overline{u}_{\xi \xi} = 0 \
\mbox{for $t<T$, with $\overline{u} = \overline{\varphi}$ at $t=T$.}
\end{equation}
(We do not need to discuss in what class the solution of \eqref{our-pde-repeated} is unique, since
our verification arguments rely on the smoothness of $u$, not its uniqueness.) By standard results about
the linear heat equation, $u$ has the required properties provided that $\overline{\varphi}$ has uniformly
bounded third derivatives, $|\overline{\varphi}_\xi| \leq c$ for all $\xi$, and
$\overline{\varphi}_{\xi \xi} \geq 0$.

Interestingly, our PDE can be reduced to the linear heat equation for a much more general class of
final-time data. We first learned this from Y. Giga. The reduction was studied in detail by Kangping
Zhu in \cite{Zhu}, and it is reviewed in an appendix to this paper. In particular, we show in the Appendix
that our PDE \eqref{our-pde-repeated} with final data $\varphi$ has a solution satisfying \eqref{require-u-smooth}--\eqref{require-time-deriv-neg} provided
\begin{gather}
\mbox{$\varphi$ is $C^4$ with uniformly bounded derivatives of order up to $4$,} \label{require-phi-smooth}\\
\mbox{$\varphi_\eta > c > 0$ for some constant $c$,} \label{require-phi-eta-pos}\\
|\varphi_\xi| \leq \varphi_\eta, \mbox{ and} \label{require-cond-on-phi-xi}\\
\varphi_{\xi \xi} \varphi_\eta^2 - 2 \varphi_{\xi \eta} \varphi_\xi \varphi_\eta +
\varphi_{\eta \eta} \varphi_\xi^2 \geq 0. \label{require-cond-neg-time-deriv}
\end{gather}

For machine learning, we are especially interested in the final-time data
$\varphi(\xi,\eta) = \frac{1}{2}(\eta + |\xi|) = \max \{x_1,x_2\}$.
This has the form \eqref{special-final-time-data} with $\overline{\varphi} = |\xi|/2$, so $u$ is determined by solving the
linear heat equation \eqref{reduces-to-heat-equation} with final-time data $|\xi|/2$. Evidently $u$ is \emph{not} uniformly
$C^4$, since its second derivatives with respect to $\xi$ blow up at $t=T$. Our method can still be applied, however, by
bounding this $\varphi$ above or below by a smooth function and considering the associated $u$.

\subsection{Summary of our main results} \label{sec2-6}

In Section \ref{sec2-4} we argued heuristically that our scaled value function
$u^\varepsilon$ is related, in the limit $\varepsilon \rightarrow 0$, to the solution
of a PDE. The rest of this paper is devoted to proving rigorous results of this type. Careful
statements will be found in the subsequent sections, but here is a brief summary:
\begin{itemize}
\item In our heuristic discussion, we eventually focused on the special case $d=1$, when the
experts' predictions depend only on the most recent market move. Our rigorous analysis also begins with
this case, since it permits us to address the essential issues in a relatively uncluttered environment.
Theorems \ref{1up} and \ref{1low} assume that the solution $u$ of our heuristic PDE \eqref{PDEsharp1a}
satisfies \eqref{require-u-smooth}--\eqref{require-ineq-btwn-xi-and-eta-derivs}.
They provide upper and lower bounds on $u^\varepsilon$, which taken together show that
$|u^\varepsilon - u| \leq C [(T-t) + \varepsilon] \varepsilon$. While the classic case $\varphi(\xi,\eta) =
\frac{1}{2}(\eta + |\xi|)$ is not covered by Theorems \ref{1up} and \ref{1low}, the additional ideas needed to
handle it were already present in \cite{Zhu}; Theorem \ref{1classic} uses them to show that for this classic case
(and still assuming $d=1$) we have $|u^\varepsilon - u| \leq C \varepsilon |\log \varepsilon|$.

\item In our heuristic discussion, the predictor's strategy had the form $f_m = f_m^* + \varepsilon f_m^\#$
(see e.g. \eqref{f_m_old}). The value of $f_m^*$ was immediately clear, but the value of $f_m^\#$ required more thought.
In Section \ref{sec2-4} we argued that the best choice of $f_m^\#$ was the one that made the investor
indifferent between the simple cycles of the relevant graph. For general $d$ we do not know whether there exists
such a choice of $f_m^\#$. But every choice of $f_m^\#$ represents a candidate strategy for the investor, and there is
a linear program that identifies the best. Our strategies for the market have a similar character: the
market ``forces'' the investor to make a particular choice of $f_m^\#$ (by penalizing her if she doesn't), and there is
a linear program that identifies the market's best choice. These linear programs are discussed in
Section \ref{LP}. For general $d$ we do not have explicit solutions of the linear programs, and we do not know
whether their optimal values match. Our understanding is, however, more complete for $d=2,3,4$; we shall
show in Section \ref{sec4-4} that these cases are like $d=1$: there is a choice of
$f_m^\#$ that achieves indifference. This choice is an explicit solution to both linear programs, and it demonstrates
that their optimal values match.

\item Theorems \ref{Dup} and \ref{Dlow} show that any admissible choice for the investor's linear program determines
a PDE-based upper bound for $u^\varepsilon$, and any admissible choice for the market's linear program determines
a PDE-based lower bound for $u^\varepsilon$. When the two linear programs have the same optimal value (which happens
at least for $d=2,3,4$) we obtain an estimate of the form
$|u^\varepsilon - u| \leq C [(T-t) + \varepsilon] \varepsilon$, entirely analogous
to the situation for $d=1$. For the classic final-time
data $\varphi = \frac{1}{2}(\eta + |\xi|)$ similar estimates hold, except that (as for $d=1$) the
error is of order $\varepsilon |\log \varepsilon|$. (This is the assertion of Theorem \ref{DclassicU}.)
\end{itemize}

Throughout this paper there are only {\it two} history-dependent experts. Subsequent to this work,
the first author and J. Calder have studied the more general case of $n$ history-dependent experts \cite{CD},
using methods quite different from those of the present work. That paper identifies
$u = \lim_{\varepsilon \rightarrow 0} u^\varepsilon$ for any number of experts $n$ and any finite $d$.
However, when the present work succeeds in determining $u$ -- specifically, when $n=2$ and $d \leq 4$ -- our
approach gives a better convergence rate ($\varepsilon$ for smooth data and
$\varepsilon |\log \varepsilon|$ for the classic case, whereas the rate obtained in
\cite{CD} is $\varepsilon^{1/3}$ and $\varepsilon^{1/3} |\log \varepsilon|$).

\section{Upper and lower bounds for $d=1$}  \label{bounds-d=1}

This section provides a fully rigorous treatment of the special case $d=1$, when the experts'
advice depends only on the most recent stock move. The upper bound is proved by considering a
good strategy for the investor -- namely, the one developed in Section \ref{sec2-4}. The
lower bound is proved by considering a good strategy for the market -- namely, the one
described at the beginning of Section \ref{sec3-2}.

\subsection{The upper bound, when $\varphi$ is regular} \label{sec3-1}
We assume throughout this section that $d=1$. Let $u^\varepsilon(t,m,\xi,\eta)$ be
defined by the dynamic programming principle \eqref{DPP} with final value
$u^\varepsilon (T,m,\xi,\eta) = \varphi(\xi,\eta)$ (it is defined only for times $t$ such
that $(T-t)/\varepsilon^2$ is an integer). Let $u(t,\xi, \eta)$ be the solution of the PDE
\begin{align} \label{PDEsharp1}
\begin{split}
&u_t+\frac{1}{2} C^\#_1\langle D^2 u \frac{\nabla^\bot u}{ u_{\eta}},
\frac{\nabla^\bot u}{ u_{\eta}} \rangle = 0,  \\
& u(T,\xi,\eta) = \varphi(\xi, \eta),
\end{split}
\end{align}
where
\begin{equation} \label{const1}
C^\#_1=\frac{(q(1)-r(1))^2+(q(0)-r(0))^2}{2}.
\end{equation}
Our goal is to prove

\begin{theorem} \label{1up}
Suppose $d=1$ and assume the PDE solution $u$ satisfies \eqref{require-u-smooth}--\eqref{require-ineq-btwn-xi-and-eta-derivs}.
Then there is a constant $C$ (independent of $\varepsilon$, $t$, and $T$) such that
\begin{equation} \label{upper-bound-statement-d=1}
u^\varepsilon(t, m, \xi, \eta)\leq u(t,\xi,\eta)+C[(T-t) + \varepsilon]\varepsilon
\quad \mbox{for $t < T$, $\xi \in \mathbb{R}$, $\eta \in \mathbb{R}$, and $m\in\{0,1\}$}
\end{equation}
whenever $\varepsilon$ is small enough and $t$ is such that $N=(T-t)/\varepsilon^2$ is an integer.
\end{theorem}
\begin{proof}
An outline of the proof is as follows: to estimate $u^\varepsilon(t_0,m_0,\xi_0,\eta_0)$, we shall
define a sequence $(t_k,m_k, \xi_k, \eta_k)$ along which $u^\varepsilon$ is monotone
$$
u^\varepsilon(t_0, m_0, \xi_0, \eta_0) \leq  u^\varepsilon(t_1, m_1, \xi_1, \eta_1)\leq
\dots \leq u^\varepsilon(t_N, m_N, \xi_N, \eta_N)
$$
with $t_N=T$, so that
$$
u^\varepsilon(t_N, m_N, \xi_N, \eta_N)=\varphi(\xi_N, \eta_N)=u(t_N, \xi_N, \eta_N).
$$
Then we'll show that $u$ is nearly constant along the sequence:
\begin{equation} \label{nearly-constant-d=1}
|u(t_0, \xi_0,\eta_0)  - u(t_N,\xi_N, \eta_N)| \leq C[ (T-t) + \varepsilon ] \varepsilon.
\end{equation}
These estimates lead immediately to \eqref{upper-bound-statement-d=1}.

To define the sequence $(t_k,m_k,\xi_k,\eta_k)$, it suffices to explain the
choice of $(t_1,m_1,\xi_1,\eta_1)$ (then the rest of the sequence is determined similarly, step
by step). Recall the dynamic programming principle \eqref{DPP}, which we can write more
compactly by defining, in terms of the investor's choice $f$,
\begin{equation} \label{defn-v}
v= \begin{pmatrix}
v^1\\
v^2
\end{pmatrix}
=
\begin{pmatrix}
q(m)-r(m)\\
q(m)+r(m)-2f
\end{pmatrix} .
\end{equation}
The dynamic programming principle then says
\begin{equation} \label{DPP-compact}
u^\varepsilon(t, m, \xi, \eta) =  \min_{|f|\leq 1} \max_{b_t=\pm 1}
u^\varepsilon(t+\varepsilon^2, m_{b_t}, \xi +\varepsilon b_t v^1, \eta+ \varepsilon b_t v^2).
\end{equation}
It follows that for any choice of $f$,
\begin{equation} \label{DPP-after-choosing-f}
u^\varepsilon(t, m, \xi,\eta) \leq \max_{b_t=\pm 1}
u^\varepsilon(t+\varepsilon^2, m_{b_t}, \xi +\varepsilon b_t v^1, \eta+ \varepsilon b_t v^2).
\end{equation}
Applying this with $(t, m, \xi,\eta)=(t_0, m_0, \xi_0,\eta_0)$ and taking $b_0$ to achieve the
max on the RHS, we see that for
$t_1=t_0 + \varepsilon^2$, $m_1=(m_0)_{b_0}$, $\xi_1= \xi_0 + \varepsilon b_0 v^1$,
$\eta_1 = \eta_0 + \varepsilon b_0 v^2$ we have the desired inequality
$u^\varepsilon(t_0, m_0, \xi_0,\eta_0) \leq u^\varepsilon(t_1, m_1, \xi_1,\eta_1)$.
While the preceding argument works for any choice of $f$, we must make a special choice
if we want the solution of the PDE to be nearly constant along the sequence. We therefore
choose $f$ as indicated by the heuristic discussion in Section \ref{sec2-4}:
\begin{equation} \label{f-when-d=1}
\begin{split}
f_{t_0,1} &= f^*_{t_0,1} + \varepsilon f^\#_{t_0,1}\\
f_{t_0,0} & = f^*_{t_0,0} + \varepsilon f^\#_{t_0,0}\\
f^*_{t_0,m} & =  \frac{(q(m)-r(m))u_{\xi}+(q(m)+r(m)) u_{\eta}}{2u_{\eta}} \\
f^\#_{t_0,1} &= f^\#_{t_0,0}=  \frac{(q(1)-r(1))^2-(q(0)-r(0))^2}{4 u_{\eta}}D.
\end{split}
\end{equation}
with the familiar convention that
\begin{equation} \label{D-reminder}
D =
\frac{1}{2}\langle D^2u \cdot \frac{\nabla^\bot u}{u_{\eta}}, \frac{\nabla^\bot u}{u_{\eta}}\rangle,
\end{equation}
and the understanding that $u_\eta$, $u_\xi$, and $D$ are evaluated at $(t_0,\xi_0,\eta_0)$.
As we explained in Section \ref{sec2-5}, this choice of $f$ is admissible (i.e. $|f|\leq 1$)
if $\varepsilon$ is sufficiently small. (This is our only smallness condition on $\varepsilon$.)
In summary: given $(t_0,m_0,\xi_0,\eta_0)$, the point $(t_1,m_1,\xi_1,\eta_1)$ is the location
that shows up on the RHS of the dynamic programming principle when $f$ is chosen by
\eqref{f-when-d=1} and $b_0$ achieves the max over $b_0=\pm 1$. The rest of the sequence
$(t_k,m_k, \xi_k, \eta_k)$ is determined similarly, step by step. (The values of $D$, $f_m^*$,
$b$, and $f_m^\#$ at step $k$ will be called $D_k$, $f_{k,m}^*$, $b_k$, and
$f_{k,m}^\#$.)

Our remaining task is to prove the near-constancy of $u$ along our sequence, \eqref{nearly-constant-d=1}.
We begin by estimating the increments
$$
U_k := u(t_{k+1}, \xi_{k+1},\eta_{k+1})-u(t_{k}, \xi_{k}, \eta_{k}).
$$
There are four cases, depending on the values of $m_k$ and $b_k$.
\medskip

\noindent \textbf{Case 1}:  $m_k=0, b_k=-1$.\\
This case is relatively easy, since the market is effectively following one of the cycles
in the $d=1$ graph (namely, the cycle $0-0$). Taylor expanding $u$ around $(t_k,\xi_k,\eta_k)$
as we did in \eqref{TaylorMinMax}, the calculation in Section \ref{sec2-4} shows that
\begin{align*}
U_k &= \varepsilon^2 \bigl[ u_t + (q(0)-r(0))^2 D_k - 2 (-1) u_\eta f_{k,0}^\# \bigr]  + O(\varepsilon^3)\\
& = \varepsilon^2 L(t_k,0, -1, \xi_k, \eta_k, f_{k,0}^\#) + O(\varepsilon^3)
\end{align*}
where for the second line we used the definition of $L$, \eqref{defn-of-L} (which reduces in this
case to \eqref{cycle00}). Moreover, the value of $L(t_k,0, -1, \xi_k, \eta_k, f_{t_k,0}^\#)$
is exactly $u_t + C_1^\# D_k$ by \eqref{time-averaged-value}, which equals zero by the PDE
\eqref{PDEsharp1}. Thus, in case $1$ we have
$$
U_k = O(\varepsilon^3).
$$
\medskip

\noindent \textbf{Case 2}:  $m_k=1, b_k=1$.\\
This case is similar to the first, since the market is effectively following another cycle in
the $d=1$ graph (namely, the cycle $1-1$). Arguing as in Case 1, we get
\begin{align*}
U_k &= \varepsilon^2 \bigl[ u_t + (q(1)-r(1))^2 D_k - 2 (+1) u_\eta f_{k,1}^\# \bigr] + O(\varepsilon^3)\\
& = \varepsilon^2 L(t_k,1, 1, \xi_k, \eta_k, f_{k,1}^\#) + O(\varepsilon^3)\\
& = O(\varepsilon^3).
\end{align*}
\medskip

\noindent \textbf{Case 3}:  $m_k=0, b_k=1$.\\
This case is different, since the market is doing just half of the two-step cycle $0-1-0$.
Starting as in the previous cases, we have
\begin{align*}
U_k &= \varepsilon^2 \bigl[ u_t + (q(0)-r(0))^2 D_k - 2 (+1) u_\eta f_{k,0}^\# \bigr]
+ O(\varepsilon^3)\\
& = \varepsilon^2 L(t_k,0,1, \xi_k, \eta_k, f_{k,0}^\#) + O(\varepsilon^3).
\end{align*}
However the value of $L$ in this case is no longer $0$. Rather, it is
\begin{multline*}
u_t + (q(0)-r(0))^2 D_k - 2 u_\eta f_{k,0}^\#  =
u_t + (q(0)-r(0))^2 D_k - \frac{1}{2} \bigl[ (q(1)-r(1))^2 - (q(0)-r(0))^2 \bigr] D_k\\
 = u_t + \frac{1}{2} \bigl[ (q(0)-r(0))^2 + (q(1)-r(1))^2 \bigr] D_k +
\bigl[ (q(0)-r(0))^2 - (q(1)-r(1))^2 \bigr] D_k.
\end{multline*}
The sum of the first two terms on the right vanishes, as a consequence of the PDE; therefore
$$
U_k = \varepsilon^2 \bigl[ (q(0)-r(0))^2 - (q(1)-r(1))^2 \bigr] D_k + O(\varepsilon^3).
$$
\medskip

\noindent \textbf{Case 4}:  $m_k=1, b_k=-1$.\\
This case is similar to the last one, since the market is again doing just half of the two-step
cycle $0-1-0$. We have
\begin{align*}
U_k &= \varepsilon^2 (u_t + (q(1)-r(1))^2 D_k - 2 (-1) u_\eta f_{k,1}^\#  + O(\varepsilon^3)\\
& = \varepsilon^2 L(t_k,1,-1, \xi_k, \eta_k, f_{k,1}^\#) + O(\varepsilon^3)
\end{align*}
and the value of $L$ this time is
\begin{multline*}
u_t + (q(1)-r(1))^2 D_k + 2 u_\eta f_{k,1}^\#  =
u_t + (q(1)-r(1))^2 D_k + \frac{1}{2} \bigl[ (q(1)-r(1))^2 - (q(0)-r(0))^2 \bigr] D_k\\
 = u_t + \frac{1}{2} \bigl[ (q(0)-r(0))^2 + (q(1)-r(1))^2 \bigr] D_k +
\bigl[ (q(1)-r(1))^2 - (q(0)-r(0))^2 \bigr] D_k.
\end{multline*}
The sum of the first two terms on the right vanishes as before, so
$$
U_k = \varepsilon^2 \bigl[ (q(1)-r(1))^2 - (q(0)-r(0))^2 \bigr] D_k + O(\varepsilon^3).
$$
(Notice that the $\varepsilon^2$ terms from cases 3 and 4 sum to zero. This had to be so, since
we know from Section \ref{sec2-4} that the average value of $L$ over the cycle $0-1-0$ is
$u_t + C_1^\# D_k$, which equals $0$.)
\medskip

We want to show that $|U_0 + \ldots U_{N-1}| \leq C [(T-t) + \varepsilon] \varepsilon$.
The $O(\varepsilon^3)$ terms in the estimates for $U_k$
are consistent with this: since there are $N = (T-t)/\varepsilon^2$ of them, they accumulate
at worst to an error that's bounded by $C (T-t) \varepsilon$. So our task is to control the sum
of the order-$\varepsilon^2$ terms from Cases 3 and 4. If the value of $D_k$ didn't change with $k$
this would be easy; but alas, it does change with $k$, since $D_k$ is the value of \eqref{D-reminder}
evaluated at $(t_k,\xi_k,\eta_k)$.

Consider the walk on the $d=1$ graph that's associated with our sequence. Suppose transitions from
$1$ to $0$ happen at steps $i_1< i_2<\dots$ and transitions from $0$ to $1$ happen at steps
$j_1< j_2<\dots$. These transitions must -- by their essential character -- be ordered. Making a choice which comes first, we may assume (without loss of generality) that $i_1 < j_1$. Depending upon which
type of transition comes last, the full list of transtions between $0$ and $1$ is either of the form
\begin{equation} \label{same-number}
i_1 < j_1 < i_2 < j_2 < \dots i_K < j_K
\end{equation}
or
\begin{equation} \label{different-number}
i_1 < j_1 < i_2 < j_2 < \dots i_K < j_K < i_{K+1}.
\end{equation}
Either way, we have
\begin{equation} \label{accumulation}
\sum_{n=1}^{K} (j_n-i_n) \leq N .
\end{equation}

If \eqref{same-number} holds, then the sum of the $\varepsilon^2$ terms (in our estimates for $U_k$)
is precisely
\begin{equation} \label{sum-of-Uk}
\varepsilon^2 \bigl[(q(1)-r(1))^2-(q(0)-r(0))^2 \bigr] \sum_{n=1}^{K} (D_{i_n} -D_{j_n}) .
\end{equation}
Now, $D_{i_n}$ and $D_{j_n}$ represent the same function \eqref{D-reminder} evaluated at different
points in space-time, which differ in time by $(j_n - i_n) \varepsilon^2$ and in space by at most
a constant times $(j_n - i_n) \varepsilon$. Our hypotheses on $u$ assure that the expression being
evaluated has uniformly bounded derivatives with respect to $t$, $\xi$, and $\eta$, so it
is globally Lipschitz continuous. We conclude that
\begin{equation} \label{wandering-of-D}
|D_{i_n} - D_{j_n}| \leq C (j_n-i_n) \varepsilon.
\end{equation}
Combining this with \eqref{accumulation} gives
$$
\varepsilon^2 \sum_{n=1}^{K} |D_{i_n} -D_{j_n}| \leq C N \varepsilon^3 = C (T-t) \varepsilon.
$$
So by \eqref{sum-of-Uk}, the $\varepsilon^2$ terms in our estimates for $U_k$
sum to at most a constant times $(T-t) \varepsilon$.

If the situation is \eqref{different-number} rather than \eqref{same-number}, the
the same argument applies, but the $\varepsilon^2$ term in $U_{i_{K+1}}$ must be handled separately.
In this case the $\varepsilon^2$ terms in our estimates for $U_k$ sum to at most a constant times
$(T-t) \varepsilon + \varepsilon^2$.

In either situation, these estimates establish \eqref{nearly-constant-d=1} , completing
the proof of the theorem.
\end{proof}

\subsection{The lower bound, when $\varphi$ is regular} \label{sec3-2}
Our lower bound shows that Theorem \ref{1up} is asymptotically sharp. Its proof is largely parallel
to that of the upper bound, except that this time we use a good strategy for the market. To describe
the strategy, recall that our heuristic discussion used Taylor expansion to estimate the
increments of $u$. As we observed just after \eqref{TaylorMinMax}, the investor must make the ``first-order term'' $\varepsilon A$ nearly vanish, since otherwise this term dominates and the
market can choose $b$ to make it positive; this determined $f_m^*$, the
leading-order term in the investor's strategy. A more subtle analysis led us to guess the
optimal next-order term, and the resulting strategy $f_m = f_m^* + \varepsilon f_m^\#$
was at the heart of our upper bound.

Our strategy for the market reflects the two-step character of the heuristic discussion:
\medskip

\noindent {\bf Case 1}: If the investor's choice doesn't nearly zero out the ``first-order term,''
then the market chooses $b$ to make that term positive; more quantitatively,
\begin{equation} \label{market-strategy-case1}
\mbox{if the investor's choice $f$ has $|f-f_m^*| \geq \gamma \varepsilon$ then the market chooses $b$
so that $-b (f-f_m^*) \geq 0$.}
\end{equation}

\noindent {\bf Case 2:} When case 1 doesn't apply, it is convenient to express the
investor's choice $f$ as
$f=f_m^* + \varepsilon f_m^\# + \varepsilon X$ (this relation defines $X$). The investor is more
optimistic than our conjectured optimal strategy if $X>0$, and more pessimistic if $X<0$.
In the former case the market makes the stock go down, and in the latter case it makes the stock
go up -- in each case giving the investor an unwelcome surprise; more quantitatively:
\begin{equation} \label{market-strategy-case2}
\mbox{if the investor's choice $f$ has $|f-f_m^*| < \gamma \varepsilon$ then the market chooses $b$
so that $-bX \geq 0$.}
\end{equation}
Here $\gamma$ is a constant, which cannot be too small; specifically, it must satisfy
\eqref{gamma-condition1} and \eqref{gamma-condition2} below. This strategy lies at
the heart of the following lower bound.
\medskip

\begin{theorem} \label{1low}
Suppose $d=1$, let $u^\varepsilon$ be defined by the dynamic programming principle
\eqref{DPP} with final-time condition $\varphi$, let $u$ solve the PDE \eqref{PDEsharp1}
with the same final-time condition, and assume $u$ satisfies
\eqref{require-u-smooth}--\eqref{require-ineq-btwn-xi-and-eta-derivs}. Then there is
a constant $C$ (independent of $\varepsilon$, $t$, and $T$) such that
\begin{equation} \label{lower-bound-statement-d=1}
u^\varepsilon(t, m, \xi, \eta)\geq u(t,\xi,\eta)- C[(T-t) + \varepsilon]\varepsilon
\quad \mbox{for $t < T$, $\xi \in \mathbb{R}$, $\eta \in \mathbb{R}$, and $m\in\{0,1\}$}
\end{equation}
whenever $\varepsilon$ is small enough and $t$ is such that $N=(T-t)/\varepsilon^2$ is an integer.
\end{theorem}
\begin{proof}
For the upper bound, we estimated $u^\varepsilon(t_0,m_0,\xi_0,\eta_0)$ by choosing a sequence
$(t_k,m_k, \xi_k, \eta_k)$ along which $u^\varepsilon$ was increasing and $u$ was nearly
constant. For the lower bound, we shall use a different sequence along which $u^\varepsilon$ is
decreasing
\begin{equation} \label{chain10-bis}
u^\varepsilon(t_0, m_0, \xi_0, \eta_0) \geq  u^\varepsilon(t_1, m_1, \xi_1, \eta_1)\geq
\dots \geq u^\varepsilon(t_N, m_N, \xi_N, \eta_N)
\end{equation}
with $t_N=T$, so that
\begin{equation} \label{chain11-bis}
u^\varepsilon(t_N, m_N, \xi_N, \eta_N)=\varphi(\xi_N, \eta_N)=u(t_N, \xi_N, \eta_N).
\end{equation}
The sequence will be chosen so that
\begin{equation} \label{chain12-bis}
u(t_N,\xi_N, \eta_N) - u(t_0, \xi_0,\eta_0) \geq - C[ (T-t) + \varepsilon ] \varepsilon.
\end{equation}
These estimates lead immediately to \eqref{lower-bound-statement-d=1}.

We shall identify the sequence by explaining the choice of $(t_1,m_1,\xi_1,\eta_1)$
(the rest of the sequence is then determined similarly, step by step). Recall our compact
form of the dynamic programming principle, equation \eqref{DPP-compact}. Let $f$ be
the investor's optimal choice at $(t_0,m_0,\xi_0,\eta_0)$; then the dynamic programming principle
becomes
\begin{equation}
u^\varepsilon(t_0, m_0, \xi_0, \eta_0) = \max_{b=\pm 1}
u^\varepsilon(t_0+\varepsilon^2, (m_{0})_b, \xi_0 +\varepsilon b v^1, \eta_0 + \varepsilon b v^2)
\end{equation}
where $v^1$ and $v^2$ are defined by \eqref{defn-v}. Evidently, either choice $b=1$ or $b=-1$
gives a point $t_1=t_0 + \varepsilon^2$, $m_1=(m_{0})_b$, $\xi_1= \xi_0 + \varepsilon v^1$,
$\eta_1 = \eta_0 + \varepsilon v^2$ for which the desired inequality
$u^\varepsilon(t_0, m_0, \xi_0,\eta_0) \geq u^\varepsilon(t_1, m_1, \xi_1,\eta_1)$ holds.
We shall show that if $b$ is chosen according to the proposed strategy
\eqref{market-strategy-case1}--\eqref{market-strategy-case2} then we obtain the desired
control on $u$.

We suppose henceforth that the sequence $(t_k, m_k, \xi_k, \eta_k)$ has been chosen by applying
the preceding argument inductively (using the proposed market strategy to determine the value of
$b_k$ at each step). Properties \eqref{chain10-bis} and \eqref{chain11-bis} are
immediately clear. Our plan for demonstrating \eqref{chain12-bis} is to show that the increments
$ U_k = u(t_{k+1}, \xi_{k+1},\eta_{k+1})-u(t_{k}, \xi_{k}, \eta_{k})$ satisfy
\begin{equation} \label{like-ub}
U_k \geq \varepsilon^2 \bigl[ u_t + (q(m_k)-r(m_k))^2 D_k - 2 b_k u_\eta f_{t_k,m_k}^\# \bigr] + O(\varepsilon^3).
\end{equation}
Crucially: the RHS of \eqref{like-ub} is the expression we found for the increment in
our proof of the upper bound.

Given \eqref{like-ub}, the rest is easy: the desired control on $u$, \eqref{chain12-bis}, follows
immediately from the argument we used for the upper bound (applied, of course, to the walk
determined by our sequence $(t_k, m_k, \xi_k, \eta_k)$).

To prepare for the proof of \eqref{like-ub}, we state now the conditions we need for the
constant $\gamma$ (recall that Case 1 of the market's strategy applies when
$|f - f_m^*| \leq \gamma \varepsilon$). Remember (from Cases 1--4 in the proof of the upper bound)
that the $\varepsilon^2$ term on the RHS of \eqref{like-ub} is either $0$ or else
\begin{equation} \label{or-else}
\pm \varepsilon^2 \bigl[(q(1)-r(1))^2-(q(0)-r(0))^2 \bigr] D_k.
\end{equation}
The expression in brackets is a fixed constant. Since we have a uniform bound on $|D_k|$ and a positive
lower bound on $u_\eta$ (by our hypotheses on $u$), by taking $\gamma$ sufficiently large we can
have
\begin{equation} \label{gamma-condition1}
\gamma \geq \bigl| (q(1)-r(1))^2-(q(0)-r(0))^2 \bigr|
\max_{\xi, \eta \in \mathbb{R}^2,~ t<T} \frac{|D|}{u_\eta}
\end{equation}
where $D$ is defined as usual by \eqref{D-reminder}. Next, recall from our heuristic discussion
that when estimating the RHS of our dynamic programming principle \eqref{DPP-compact} by Taylor expansion, the second-order-term is
$$
\varepsilon^2 \bigl[ u_t + \frac{1}{2} \langle D^2u \cdot v_k, v_k \rangle \bigr]
$$
where $v_k$ is defined by \eqref{defn-v}. Since the investor must choose $|f|\leq 1$, we have
a uniform bound $|v_k|\leq M$. Since we are assuming uniform bounds for $u_t$ and $D^2 u$ (and recalling
that $u_\eta$ is bounded away from $0$) by taking $\gamma$ sufficiently large we can have
\begin{equation} \label{gamma-condition2}
\gamma \geq
\max_{\substack{\xi, \eta \in \mathbb{R}^2,~ t<T\\|v|\leq M}}
\frac{\bigl| u_t + \frac{1}{2} \langle D^2u \cdot v, v \rangle \bigr|}{u_\eta}.
\end{equation}
These are the only conditions we place on $\gamma$.

We now prove \eqref{like-ub} for the market's Case 1. At step $k$ (when the current state is
$m_k$) this means the investor's choice $f_k$ has $|f_k-f_{k,m_k}^*| \geq \gamma \varepsilon$.
(Here we write $f_k$ and $f_{k,m_k}^*$ rather than $f_{t_k}$ and $f_{t_k,m_k}^*$ to simplify the notation.) Proceeding as we
did in the heuristic discussion, the increment is
\begin{equation} \label{markets-case1-step1}
U_k = \varepsilon A + \varepsilon^2 B + O( \varepsilon^3)
\end{equation}
with
$$
A= b_k (v^1 u_\xi + v^2 u_\eta) =
b_k \bigl( [q(m_k)-r(m_k)]u_{\xi}+[q(m_k)+r(m_k)-2f_k]u_{\eta} \bigr)
$$
and
$$
B = u_t + \frac{1}{2} \langle D^2u \cdot v_k, v_k \rangle .
$$
Since $A$ is an affine function of $f_k$ and it vanishes when $f_k=f_{k,m_k}^*$, the
expression for $A$ simplifies to
\begin{equation} \label{markets-case1-step2}
A=-2 b_k u_\eta (f_k-f_{k,m_k}^*).
\end{equation}
Since we are in Case 1, the market chooses $b_k$ so that $A$ is positive, and we have
$$
\varepsilon A \geq \varepsilon^2 [ 2 \gamma u_\eta].
$$
Since $B \leq \gamma u_\eta$ by \eqref{gamma-condition2}, we have
$$
\varepsilon A + \varepsilon^2 B \geq \varepsilon^2 [\gamma u_\eta].
$$
Making use now of \eqref{gamma-condition1}, we conclude that
$$
\varepsilon^2 [\gamma u_\eta ] \geq
\varepsilon^2 \bigl|(q(1)-r(1))^2-(q(0)-r(0))^2 \bigr| D_k;
$$
in view of \eqref{or-else}, this implies the desired estimate \eqref{like-ub}.

Turning now to the market's Case 2, we start again with \eqref{markets-case1-step1}. The value
of $A$ is again given by \eqref{markets-case1-step2}. Writing
$f_k - f_{k,m_k}^* = \varepsilon f_{k,m_k}^\# + \varepsilon X$ and choosing $b_k$ by
the market's strategy \eqref{market-strategy-case2}, we have
$$
\varepsilon A \geq \varepsilon^2 [- 2 b_k u_\eta f_{k,m_k}^\#] .
$$
Since $|f_k - f_{k,m_k}^*| \leq \gamma \varepsilon$, the ``second-order term''
$B = u_t + \frac{1}{2} \langle D^2u \cdot v_k, v_k \rangle$ can be estimated
as we did in the heuristic calculation: replacing $f_k$ by $f_{k,m_k}^*$ in the
expression for $v_k$ makes an error of order $\varepsilon$, and this leads to
$$
\varepsilon^2 B = \varepsilon^2 [u_t + (q(m)-r(m))^2 D_k] + O(\varepsilon^3).
$$
Adding, we conclude that
\begin{align*}
U_k &= \varepsilon A + \varepsilon^2 B + O(\varepsilon^3) \\
 & \geq \varepsilon^2 \bigl[u_t + (q(m)-r(m))^2 D_k -  2 b_k u_\eta f_{k,m_k}^\# \bigr] +
 O(\varepsilon^3),
\end{align*}
which is precisely \eqref{like-ub}.
The proof of the theorem is now complete.
\end{proof}

\subsection{The classic case, when $\varphi= (\eta + |\xi|)/2$} \label{sec3-3}
The classic goal of minimizing regret with respect to the best-performing expert
corresponds, as we explained in Section \ref{sec2-3}, to using final-time data
$\varphi = (\eta + |\xi|)/2$. Since this $\varphi$ is not $C^3$, the solution of the
associated PDE does not admit a uniform $C^3$ bound up to $t=T$, so
Theorems \ref{1up} and \ref{1low} do not include this case. However we can obtain
similar estimates (with an error estimate of order $\varepsilon |\log \varepsilon|$ rather
than $\varepsilon$) by repeating the proofs of those theorems with proper attention
to the error terms. A result of this type was proved by Zhu for two constant experts \cite{Zhu};
the proof of the following theorem uses essentially the same arguments.

\begin{theorem} \label{1classic}
Suppose $d=1$ and consider the classic final-time data $\varphi = (\eta + |\xi|)/2$. We consider, as usual,
the function $u^\varepsilon(t,m,\xi,\eta)$ defined by the dynamic programming principle \eqref{DPP} with final-time data $\varphi$, and the function $u(t,\xi,\eta)$ defined by the PDE \eqref{PDEsharp1} with
final-time data $\varphi$. There is a constant $C$ (independent of $\varepsilon$, $t$, and $T$) such that
\begin{equation} \label{estimate-classic-d=1}
\bigl| u^\varepsilon(t, m, \xi, \eta) - u(t,\xi,\eta) \bigr| \leq
C \varepsilon |\log \varepsilon|
\quad \mbox{for $t < T$, $\xi \in \mathbb{R}$, $\eta \in \mathbb{R}$, and $m\in\{0,1\}$}
\end{equation}
whenever $\varepsilon$ is small enough and $t$ is such that $N=(T-t)/\varepsilon^2$ is an integer.
\end{theorem}

\begin{proof}
As we observed in Section \ref{sec2-5}, the PDE simplifies dramatically in this case, and the solution
is $u(t,\xi,\eta) = \frac{1}{2} \eta + \overline{u}(t,\xi)$ where $\overline{u}$ solves
the linear heat equation $\overline{u}_t + \frac{1}{2}C_1^\# \overline{u}_{\xi \xi} = 0$ for $t<T$ with
$\overline{u} = \frac{1}{2} |\xi|$ at $t=T$. We note that the ``diffusion constant'' $\frac{1}{2}C_1^\#$ (which is
determined by \eqref{PDEsharp1b}) is nonzero, since we assumed in \eqref{q-and-r-distinct} that the experts were
distinct. Using the explicit solution formula for the linear heat equation we have
\begin{equation} \label{estimates-for-u}
|\partial^k_\xi \partial^l_t \overline{u}| \leq C_{k,l} (T-t)^{-(k+2l-1)/2}   \quad
\mbox{for $k \geq 0$ and $l \geq 0$ with $k + 2l \geq 1$}.
\end{equation}

Our overall strategy is to repeat the arguments used for Theorems \ref{1up} and \ref{1low}, with $u(t,\xi,\eta)$
replaced by the smooth function $\tilde{u}(t,\xi,\eta) = u(t-\delta, \xi, \eta)$
(which satisfies the same PDE, with final-time data that's a smooth approximation to
$\varphi$). We will choose the value of $\delta$ below to optimize
the resulting estimate.

To see what emerges from this strategy, we need to revisit the various
``error terms'' that entered the proofs of the theorems, namely:
\begin{enumerate}
\item[(i)] those incurred by estimating the increments of $u$ using Taylor expansion, i.e.
the $O(\varepsilon^3)$ term in the estimate
$$
u(t_{k+1}, \xi_{k+1},\eta_{k+1})-u(t_{k}, \xi_{k}, \eta_{k}) =
\varepsilon A + \varepsilon^2 B + O(\varepsilon^3)
$$
where $A$ and $B$ are given by \eqref{defn-of-A} and \eqref{defn-of-B} evaluated at step $k$;

\item[(ii)] those made by evaluating $B$ at $f_{t_k,m_k}^*$ rather than $f_{t_k,m_k}$, i.e.
the $O(\varepsilon^3)$ term in \eqref{value-of-B};

\item[(iii)] our estimate \eqref{wandering-of-D} for $|D_{i_n} - D_{j_n}|$; and

\item[(iv)] our estimate for the value of $D_{K+1}$, which was needed when analyzing a walk
with a different number of transitions from $0$ to $1$ vs $1$ to $0$ (i.e. when
\eqref{different-number} applies).
\end{enumerate}
We also need to monitor errors associated with
\begin{enumerate}
\item[(v)] the difference $|u(t-\delta,\xi,\eta) - u(t,\xi,\eta)|$, evaluated both at the time $t$ that
enters the estimate, and at the final time $T$ when $u(T,\xi,\eta)=\varphi (\xi,\eta)$.
\end{enumerate}
We begin with some easy estimates, which are too crude to give \eqref{estimate-classic-d=1}
but already give a nontrivial result. For any $\delta > 0$ we have uniform bounds
on the derivatives of $\tilde{u}(t,\xi,\eta) = u(t-\delta,\xi,\eta)$ for $t<T$, obtained by taking
$(T-t)=\delta$ in \eqref{estimates-for-u}. Applying this:
\begin{enumerate}
\item[(i)] The errors of type (i) are most easily estimated by writing the increment
$u(t_{k+1},\xi_{k+1},\eta_{k+1}) -  u(t_k,\xi_k,\eta_k)$ as
$$
[ u(t_{k+1},\xi_{k+1},\eta_{k+1}) -  u(t_{k+1},\xi_k,\eta_k) ] +
[ u(t_{k+1},\xi_k,\eta_k) -  u(t_k,\xi_k,\eta_k) ],
$$
then estimating each term separately using Taylor's theorem. If $\delta^{-1} \varepsilon^2 \leq 1/2$, the error
due to estimating a single increment by Taylor expansion is of order $\delta^{-1} \varepsilon^3$.
There are $N=(T-t)/\varepsilon^2$ such increments, so these accumulate to a term of order
$\delta^{-1} (T-t) \varepsilon$.

\item[(ii)] Since the spatial second derivatives of $\tilde{u}$ are at most of order $\delta^{-1/2}$, the errors
of type (ii) in the proof of the upper bound are at most $\delta^{-1/2} \varepsilon^3$ at each time step, accumulating to
$\delta^{-1/2} (T-t) \varepsilon$ after summation over all steps. The situation is slightly worse in the proof
of the lower bound, since in the market's Case 2 we only have $|f_{t,m} - f^*_{t,m}| \leq \gamma \varepsilon$ and
the size of $\gamma$ is controlled by \eqref{gamma-condition1}--\eqref{gamma-condition2}. These conditions
can be met with $\gamma \sim \delta^{-1/2}$; this leads to a type-(ii) error of order $\delta^{-1} \varepsilon^3$
at each time step, accumulating to $\delta^{-1} (T-t) \varepsilon$ after summation over all steps.

\item[(iii)] If $\delta^{-1} \varepsilon^2 \leq 1/2$ then the increment in the value of $D_k$ from one
step to the next is at most of order $\delta^{-1} \varepsilon$. This gives a substitute for \eqref{wandering-of-D}, namely
$|D_{i_n}-D_{j_n}| \leq C\delta^{-1}\varepsilon (j_n - i_n) $.

\item[(iv)] Since each $D_k$ is at most of order $\delta^{-1/2}$, a single term
$\varepsilon^2 D_{K+1}$ is estimated by $\delta^{-1/2} \varepsilon^2$.

\item[(v)] Since $u(t-\delta,\xi,\eta) - u(t,\xi,\eta) = \int_{t-\delta}^t u_s(s,\xi,\eta) \, ds$,
the value at $t=T$ is of order $\delta^{1/2}$, and the value at an earlier time $t$ is not larger.
\end{enumerate}
Arguing as for Theorems \ref{1up} and \ref{1low} and using these estimates, one finds that
$$
\bigl| u^\varepsilon(t, m, \xi, \eta) - u(t,\xi,\eta) \bigr| \leq
C \bigl( \delta^{-1} (T-t) \varepsilon + \delta^{-1/2} \varepsilon^2 + \delta^{1/2} \bigr)
$$
provided $\delta^{-1} \varepsilon^2 \leq 1/2$. To make the first and last terms similar it is
natural to take $\delta \sim (T-t)^{2/3} \varepsilon^{2/3}$. (This is consistent with $\delta^{-1} \varepsilon^2 \leq 1$
since $(T-t) \geq \varepsilon^2$.) The middle term is then of order $(T-t)^{-1/3} \varepsilon^{5/3}$, which is dominated
by the other two; thus we conclude that $|u^\varepsilon - u| \leq C (T-t)^{1/3} \varepsilon^{1/3} $.

The preceding estimates are too crude, because they estimate the errors at \emph{every} time step using
uniform bounds for the derivatives of $\tilde{u}$, which are overly pessimistic when $t \ll T$. To do better,
we should use a $k$-dependent estimate for the errors at time $t_k$. Since $\tilde{u}(t,\xi,\eta) = u(t-\delta,\xi,\eta)$,
the derivatives of $\tilde{u}$ are estimated at time $t_k$ by \eqref{estimates-for-u} with $T-t$
replaced by $T-t_k + \delta$. For errors that accumulate over many time steps we must sum
the resulting series. This is easily done by comparison to a suitable integral: as
$t_k$ runs from $t$ to $T$ in increments of $\varepsilon^2$, $\tau_k = T-t_k + \delta$ runs from
$T-t+\delta$ to $\delta$, and a sum of the form $\sum_{k=0}^N \tau_k^{-a} \varepsilon^2$ is essentially a Riemann sum for $\int_\delta^{T-t+\delta} s^{-a} \, ds$; as a result we have
\begin{equation} \label{riemann-sum-estimate}
\sum_{k = 0}^N \tau_k^{-a} \leq \left\{
\begin{array}{ll}
C_a \varepsilon^{-2} \delta^{1-a} & \mbox{when $a > 1$, and}\\
C \varepsilon^{-2} |\log \delta| & \mbox{when $a = 1$}.
\end{array}
\right.
\end{equation}

\noindent We now review the ``error terms'' from this perspective:

\begin{enumerate}

\item[(i)] As in the previous calculation, we write
$\tilde{u}(t_{k+1},\xi_{k+1},\eta_{k+1}) -  \tilde{u}(t_k,\xi_k,\eta_k)$ as
$$
[ \tilde{u}(t_{k+1},\xi_{k+1},\eta_{k+1}) -  \tilde{u}(t_{k+1},\xi_k,\eta_k) ] +
[ \tilde{u}(t_{k+1},\xi_k,\eta_k) -  \tilde{u}(t_k,\xi_k,\eta_k) ].
$$
When we estimate the first term by $2$nd order Taylor expansion in space and the second by
$1$st order Taylor expansion in time, we introduce an error of order
$$
\tau_{k+1}^{-1} \varepsilon^3 + \tau_{k+1}^{-3/2} \varepsilon^4
$$
by \eqref{estimates-for-u}. Using \eqref{riemann-sum-estimate} and assuming $\delta^{-1} \varepsilon^2 \leq 1/2$,
we see that these errors sum to at most a constant times $|\log \delta | \varepsilon + \varepsilon$.

\item[(ii)] The type-(ii) error in the upper-bound argument is
controlled at step $k$ by $|\tilde{u}_{\xi \xi}| \varepsilon^3  \leq \tau_k^{-1/2} \varepsilon^3.$
The situation is slightly worse in the lower bound argument, since the constant $\gamma$ in the market's
Case 2 scales like $\tau_k^{-1/2}$ at step $k$, leading to a type-(ii) error of order
$\tau_k^{-1} \varepsilon^3$. By \eqref{riemann-sum-estimate}, these errors sum to a term of order at most
$|\log \delta| \varepsilon$.

\item[(iii)] The increment in the value of $D_k$ from one step to the next is estimated by
$\varepsilon^2 |\partial_t D|  + \varepsilon |\partial_\xi D|$. At step $k$ this is at most
$\varepsilon^2 \tau_k^{-3/2} + \varepsilon \tau_k^{-1}$. Using \eqref{riemann-sum-estimate}
and assuming $\delta^{-1} \varepsilon^2 \leq 1/2 $, we see that the sum of these errors is at most
of order $|\log \delta| \varepsilon + \varepsilon$.

\item[(iv)] At time $k$, a single term $\varepsilon^2 D_{k}$ is estimated by
$\tau_k^{-1/2} \varepsilon^2$. Since $\tau_k \geq \delta$, when $\delta^{-1} \varepsilon^2 \leq 1/2$
this term is of order $\varepsilon$.

\item[(v)] Our previous estimate of the type-(v) error remains adequate for our present purpose: it is
at most $\delta^{1/2}$.
\end{enumerate}
Thus: both the upper-bound and lower-bound arguments, applied using $\tilde{u}$ rather than $u$,
give
$$
|u^\varepsilon(t,m,\xi,\eta) - u(t,\xi,\eta)| \leq
C \bigl( \varepsilon |\log \delta| + \varepsilon + \sqrt{\delta} \bigr)
$$
provided $\delta^{-1} \varepsilon^2 \leq 1/2$. Choosing $\delta = 2\varepsilon^2 $, we obtain the
desired estimate \eqref{estimate-classic-d=1}.
\end{proof}

We note that for the classic final-time data our estimate \eqref{estimate-classic-d=1} is independent of
time, while in Theorems \ref{1up} and \ref{1low} it was proportional to $(T-t)$. This difference arises
because for the classic final-time data, $u_\eta$ is constant and all the $\xi$ derivatives of $u$
of order $2$ or more decay to $0$ as $T-t \rightarrow \infty$ (note that $u_\xi$ solves the linear heat
equation with final time data $\half \mbox{sgn} \, \xi$). In the more general setting of Theorems
\ref{1up}--\ref{1low} the derivatives driving the error terms are controlled, but it is not clear that
they tend to $0$ as $T-t \rightarrow \infty$.

\section{The linear programs} \label{LP}
We have thus far focused on experts that use only the most recent market move (the case $d=1$). We identified
a strategy for the investor that makes her indifferent with respect to the cycles on the $d=1$ graph
(Section \ref{sec2-4}); then we used that strategy to prove upper and lower bounds that match at leading order
(Theorems \ref{1up} and \ref{1low}).

The situation is similar for experts that use up to four recent market moves (the cases $d=2,3,4$): there is a
strategy for the investor that achieves indifference with respect to the cycles on the relevant graph (see Section
\ref{sec4-4}), and it leads to upper and lower bounds that match at leading order (see Remark \ref{remark-on-matching-bounds}
in Section \ref{sec5-2}).

For experts that use more history ($d \geq 5$) we do not know whether indifference is achievable. Our methods still lead to
upper and lower bounds, but it is no longer clear that they match at leading order. The upper bound is associated with
a strategy for the investor that minimizes the maximum rate at which regret accumulates, among all possible cycles
on the graph. The lower bound is associated with a strategy for the market that maximizes the minimum rate at which regret accumulates, among all possible cycles on the graph. The identification of these strategies is the focus of this section.

In the proofs of Theorems \ref{1up} and \ref{1low}, most of the work involved considering how the solution $u$ of our
PDE changed along a well-chosen path $(t_0, m_0, \xi_0 , \eta_0), (t_1, m_1, \xi_1 , \eta_1), \ldots$ Within the special class of
strategies described by \eqref{f_m_old} and \eqref{f^*m}, we showed in Section \ref{sec2-4} that the increments are
$$
u(t_{k+1},\xi_{k+1}, \eta_{k+1}) - u(t_k, \xi_k, \eta_k) = \eps^2 L(t_k, m_k, b_k, \xi_k, \eta_k, f^\#_{t_k,m_k}) + O(\eps^3)
$$
where
$$
L(t, m, b, \xi, \eta, f^\#_{t, m}) = u_t +(q(m)-r(m))^2 D - 2 b u_{\eta} f^\#_{t, m}
$$
(see \eqref{TaylorMinMax} -- \eqref{defn-of-L}); here we use our usual convention
$D = \half \langle D^2u \cdot \frac{\nabla^\bot u}{u_{\eta}}, \frac{\nabla^\bot u}{u_{\eta}}\rangle$,
and $b_k$ is determined by the relation $(m_k)_{b_k} = m_{k+1}$. Since $u$ represents our estimate of
the investor's worst-case regret, we think of $ \eps^2 L$ as the {\it increment of regret}.

Our linear programs are concerned with the average rate at which regret accumulates, as the state $m$ traverses
various cycles. In formulating them, we will treat $u_t$, $D$, and $u_\eta$ as constants, ignoring the fact
that they are functions evaluated at different points in space-time. This seems reasonable, since the state changes
at every step while the location in space-time changes slowly (by increments of order $\eps$ in space and
$\eps^2$ in time). Our rigorous bounds, presented in Section \ref{bounds-d>1}, will
of course take into account the fact that $u_t$, $D$, and $u_\eta$ are not really constant.

As an introduction to the linear programs, it is convenient to revisit the case $d=1$. Its graph, shown in Figure
\ref{fig:graph(d=1)withoutcosts}, has three cycles: $00$, $11$, and $010$. The condition that for each cycle, the average
rate at which regret accumulates is at most $\eps^2 R$ is therefore
\begin{equation}
\begin{gathered}
u_t + (q(0) - r(0))^2 D + 2 u_\eta f_0^\# \leq R \\
u_t + (q(1) - r(1))^2 D - 2 u_\eta f_1^\# \leq R \label{d=1-lp-inv} \\
2 u_t + [(q(0)-r(0))^2 + (q(1)-r(1))^2] D - 2 u_\eta f_0^\# + 2 u_\eta f_1^\# \leq 2R.
\end{gathered}
\end{equation}
The investor wants to make $R$ as small as possible. Since $u_t$ is being treated as a constant, we
can move it to the right hand side. Setting $M = (R-u_t)/D$ and
\begin{equation} \label{f_m^sharp-vs-beta_m}
\beta_m = \frac{2 u_\eta}{D} f_m^\# ,
\end{equation}
and assuming $D > 0$, we see that \eqref{d=1-lp-inv} is equivalent to
\begin{equation}
\begin{gathered}
(q(0) - r(0))^2  + \beta_0 \leq M \\
(q(1) - r(1))^2  - \beta_1 \leq M  \label{d=1-lp-reduced-inv} \\
[(q(0)-r(0))^2 + (q(1)-r(1))^2] - \beta_0 + \beta_1 \leq 2M.
\end{gathered}
\end{equation}
(In practice, the parameter $D$ will come from the solution of our PDE, and it is
nonnegative by \eqref{require-time-deriv-neg}. We need not be concerned about the exceptional
case $D=0$, which is handled in Section \ref{bounds-d>1} by taking $f_m^\# = 0$.)
The {\it investor's linear program} for $d=1$ is thus to find $\beta_0$, $\beta_1$, and $M$
that minimize $M$ subject to \eqref{d=1-lp-reduced-inv}.

A similar discussion applies for the market. The condition that for each simple cycle, the average
rate at which regret accumulates is at least $\eps^2 R$ is
\begin{gather*}
u_t + (q(0) - r(0))^2 D + 2 u_\eta f_0^\# \geq R \\
u_t + (q(1) - r(1))^2 D - 2 u_\eta f_1^\# \geq R \\
2 u_t + [(q(0)-r(0))^2 + (q(1)-r(1))^2] D - 2 u_\eta f_0^\# + 2 u_\eta f_1^\# \geq 2R.
\end{gather*}
Changing variables as before, this is equivalent to
\begin{equation}
\begin{gathered}
(q(0) - r(0))^2  + \beta_0 \geq M \\
(q(1) - r(1))^2  - \beta_1 \geq M \label{d=1-lp-reduced-mkt}\\
[(q(0)-r(0))^2 + (q(1)-r(1))^2] - \beta_0  + \beta_1 \geq 2M.
\end{gathered}
\end{equation}
The {\it market's linear program} for $d=1$ is thus to find $\beta_0$, $\beta_1$, and $M$
that maximize $M$ subject to \eqref{d=1-lp-reduced-mkt}.

The following properties of these linear programs are immediately evident:
\begin{enumerate}
\item[(i)] If $\beta_0$, $\beta_1$, and $M$ are admissible for \eqref{d=1-lp-reduced-inv}, then
adding the constraints gives
$$
M \geq \frac{(q(0)-r(0))^2 + (q(1)-r(1))^2}{2}.
$$
\item[(ii)] If $\beta_0$, $\beta_1$, and $M$ are admissible for \eqref{d=1-lp-reduced-mkt}, then
adding the constraints gives
$$
M \leq \frac{(q(0)-r(0))^2 + (q(1)-r(1))^2}{2}.
$$
\item[(iii)] In view of (i) and (ii), the values of the two linear programs coincide when
there is a choice of $\beta_0$ and $\beta_1$ that achieves indifference, in the sense that the LHS
of each of the inequalities in \eqref{d=1-lp-reduced-inv} (or equivalently, each of the inequalities
in \eqref{d=1-lp-reduced-mkt}) takes the same value; moreover the common value is then
$\half [(q(0)-r(0))^2 + (q(1)-r(1))^2]$.
\item[(iv)] The values of the two linear programs do indeed coincide, since the choice
$\beta_0 = \beta_1 = \half[(q(1)-r(1))^2 - (q(0)-r(0))^2]$ achieves indifference.
\end{enumerate}
We shall show in Sections \ref{sec4-1} and \ref{sec4-2} that analogues of (i)--(iii) hold for
any $d$. We do not know whether the analogue of (iv) holds for any $d$, however we shall show
in Section \ref{sec4-4} that indifference is achievable when $d=2$, $3$, or $4$.

\subsection{The investor's linear program} \label{sec4-1}

The investor's linear program for $d=1$, presented above as the minimization of $M$ subject to \eqref{d=1-lp-reduced-inv},
can be written as
\begin{align} \label{d=1-lp-final-inv}
&\min M ~~~ \mbox{such that}\nonumber \\
&  A\begin{pmatrix}
-\beta_0\\
-\beta_1\\
M
\end{pmatrix} \leq g
\end{align}
with
$$
A =
\begin{pmatrix}
 -1& 0& -1& \\
 0 &1 & -1\\
1& -1&  -2
\end{pmatrix}
\quad \mbox{and} \quad
g= -
\begin{pmatrix}
(q(0)-r(0))^2 \\ (q(1)-r(1))^2 \\ (q(0)-r(0))^2 + (q(1)-r(1))^2
\end{pmatrix}
.
$$

Notice that each row of $A$ corresponds to a simple cycle; the final entry in a row is minus the length of the cycle, while
the row's other elements reflect the details of the cycle. As for $g$: each row is minus the sum of $(q(m)-r(m))^2$, as
$m$ ranges over the vertices participating in the associated cycle.

The investor's linear program is similar for any value of $d$: it always has the form
\begin{align} \label{LP1m}
&\min M ~~~ \mbox{such that}\nonumber \\
&  A\begin{pmatrix}
-\beta_{0}\\
\dots\\
-\beta_{ 2^d-1}\\
M
\end{pmatrix} \leq g
\end{align}
where each row of $A$ corresponds to a simple cycle on the underlying graph (the $d$-dimensional de Bruijn graph on
$2$ symbols, see Section \ref{sec2-2}). If the graph has $l$ simple cycles then $A$ is an $l \times (2^d + 1)$ matrix.
Its first $2^d$ columns are in correspondence with the vertices: for $0 \leq m \leq 2^d -1$, the $(m+1)$th column reports,
for each cycle, whether the cycle includes vertex $m$, and if so then whether the cycle leaves it
by a $+$ edge or a $-$ edge. The last column of $A$ is, up to sign, the number of edges in the cycle. The elements
of $g$ are, up to sign, the sum of $(q(m)-r(m))^2$ as $m$ ranges over the vertices in the given cycle. More explicitly:
the first $2^d$ columns of the matrix $A$ are determined by
\begin{align} \label{first-cols-of-A}
a_{i,m+1}=1 & \quad \mbox{if cycle $i$ contains an edge from $m$ to $m_+$}; \nonumber \\
a_{i,m+1}=-1 & \quad \mbox{if cycle $i$ contains an edge from $m$ to $m_-$};\\
a_{i,m+1} = 0 & \quad \mbox{if cycle $i$ does not pass through vertex $m$;} \nonumber
\end{align}
the last column of $A$ has entries
\begin{equation} \label{last-col-of-A}
a_{i, 2^d+1}=-|s_i| \ \mbox{where $|s_i|$ is the number of edges in cycle $i$;}
\end{equation}
and the elements of $g$ are
\begin{equation} \label{entries-of-g}
g_i =  -\sum_{{\rm vertices} \, m \, {\rm in \, cycle} \, i}(q(m)-r(m))^2.
\end{equation}
While $A$ happens to be square when $d=1$, it is usually rectangular; for example, when $d=2$ there are $6$ simple
cycles and $4$ vertices, so $A$ is $6 \times 5$.

The interpretation of the investor's linear program \eqref{LP1m} is the same for any $d$ as it was for $d=1$:
if $\beta_m$ ($0 \leq m \leq 2^d - 1$) and $M$ are admissible then choice $f_m^\# = \frac{D}{2u_\eta} \beta_m$
assures that the average rate at which regret accumulates over any simple cycle is at most $\eps^2 R$ with $R = u_t + MD$.

\begin{lemma} \label{lemma:investor-LP}
The investor's linear program \eqref{LP1m} has the following properties:
\begin{enumerate}
\item[(a)] The feasible set is nonempty; moreover, every feasible point has the property that
\begin{equation} \label{lower-bound-on-M}
M \geq \frac{\sum_{m =0}^{2^d-1}(q(m)-r(m))^2}{2^d}.
\end{equation}
\item[(b)] The optimum is achieved.
\item[(c)] If there is a strategy that achieves indifference (i.e. if there is a feasible point for which the constraints
all hold with equality) then the optimal value of the investor's linear program is\\
$ 2^{-d} \sum_{m =0}^{2^d-1}(q(m)-r(m))^2$.
\end{enumerate}
\end{lemma}
\begin{proof}
The existence of a feasible point is easy: if we take $\beta_m = 0$ for all $m$, then the constraints are satisfied provided
that for each cycle $i$,
$$
M \geq \frac{1}{|s_i|} \sum_{{\rm vertices} \, m \, {\rm in \, cycle} \, i}(q(m)-r(m))^2
$$
where $|s_i|$ is the number of edges in cycle $i$.

To complete the proof of (a) we must establish the lower bound \eqref{lower-bound-on-M}. We use here the fact that
our de Bruijn graph is {\it Eulerian}, i.e. there exists a closed walk on the graph that traverses each edge exactly once
(see e.g. \cite{West}). Consider a decomposition of this walk as a union of simple cycles.
Clearly no cycle can appear more than
once; and in the union of all the cycles that are used, each vertex $m$ appears twice -- once in connection with the edge
from $m$ to $m_+$, and once in connection with the edge from $m$ to $m_-$. When we add the constraints corresponding the
cycles that are used, the terms involving $\beta_m$ cancel (since each $\beta_m$ appears twice, with opposite signs -- once
due to the edge from $m$ to $m_+$, and once due to the edge from $m$ to $m_-$). The
sum of these constraints thus reduces to
$$
2 \sum_{m=0}^{2^d-1} (q(m)-r(m))^2 \leq 2^{d+1} M,
$$
which is equivalent to \eqref{lower-bound-on-M}.

Assertion (b) follows immediately from (a), using the general result from linear programming that if an LP
of the form \eqref{LP1m}
is feasible and bounded below then the optimal value is achieved. Alternatively (and more constructively),
one can see (b) by considering the simplex method (with a suitable scheme to prevent cycling). At each step
the simplex method moves in a direction that improves the objective. A move to infinity cannot occur, since
the objective is bounded below. Therefore the simplex method moves from vertex to vertex, terminating
at one that is optimal.

Assertion (c) is clear from the proof of \eqref{lower-bound-on-M}. Indeed, we proved that result by adding the inequalities
associated with the cycles of an Eulerian circuit. A strategy that achieves indifference turns those inequalities into
equalities, so equality must also hold in \eqref{lower-bound-on-M}.
\end{proof}

\subsection{The market's linear program} \label{sec4-2}

Comparing \eqref{d=1-lp-reduced-inv} and \eqref{d=1-lp-reduced-mkt}, we see that for $d=1$ the market's linear program
is obtained from the investor's by changing the direction of the inequalities and changing the objective from
$\min M$ to $\max M$. The situation is the same for any $d$: the market's linear program is
\begin{align} \label{LP2m}
&\max M ~~~ \mbox{such that}\nonumber \\
&  A\begin{pmatrix}
-\beta_{0}\\
\dots\\
-\beta_{ 2^d-1}\\
M
\end{pmatrix} \geq g
\end{align}
where $A$ and $g$ are still defined by \eqref{first-cols-of-A}--\eqref{entries-of-g}. Its interpretation is analogous
to the case $d=1$: if $\beta_m$ ($0 \leq m \leq 2^d - 1$) and $M$ are admissible for \eqref{LP2m} then the choice
$f_m^\# = \frac{D}{2u_\eta} \beta_m$ assures that the average rate at which regret accumulates over
any simple cycle is at least $\eps^2 R$ with $R = u_t + MD$.

The following analogue of Lemma \ref{lemma:investor-LP} is proved using the same arguments.

\begin{lemma} \label{lemma:market-LP}
The market's linear program \eqref{LP2m} has the following properties:
\begin{enumerate}
\item[(a)] The feasible set is nonempty; moreover, every feasible point has the property that
\begin{equation} \label{upper-bound-on-M}
M \leq \frac{\sum_{m =0}^{2^d-1}(q(m)-r(m))^2}{2^d}.
\end{equation}
\item[(b)] The optimum is achieved.
\item[(c)] If there is a strategy that achieves indifference (i.e. if there is a feasible point for which the constraints
all hold with equality) then the optimal value of the market's linear program is\\
$ 2^{-d} \sum_{m =0}^{2^d-1}(q(m)-r(m))^2$.
\end{enumerate}
\end{lemma}

\begin{corollary} \label{corollary:indifference-implies-equal-values}
The optimal value of the market's linear program is less than or equal to that of the investor's linear program. If
there is a strategy that achieves indifference then the optimal values are equal.
\end{corollary}

\begin{proof}
The first assertion is clear by combining part (a) of Lemma \ref{lemma:investor-LP} with part (a) of Lemma \ref{lemma:market-LP}.
The second assertion follows from part (c) of each Lemma.
\end{proof}

\subsection{Why are the LP's not convex duals?} \label{sec4-3}

In finite dimensions, zero-sum two-person games using mixed strategies can
be analyzed using the duality theory of linear programming, leading to a
pair of dual linear programs -- one for the ``minimizing'' player, the other
for the ``maximizing'' player. We emphasize that this is \textit{not} the relationship
between our linear programs \eqref{LP1m} and \eqref{LP2m}.

It is natural to ask whether duality might somehow be relevant to our problem. The answer
is this:
\begin{enumerate}
\item[(a)] If the rules of the game {\it required} the investor to choose
$f_m = f_m^* + \varepsilon f_m^\#$, as envisioned by the heuristic argument presented in
Section \ref{sec2-4}, then duality would be relevant (at least formally). Indeed, as we shall show in a
moment, the dual of the investor's linear program provides a convex combination of simple cycles that makes the
market indifferent to the investor's choice of $f_m^\#$ and provides (at least formally) a lower bound that
matches our upper bound.

\item[(b)] Alas, the rules of our game do not require the investor to behave this way (and we have
not proved that the optimal choice has this form). In our lower bound for $d \geq 2$, presented in Section
\ref{bounds-d>1}, the idea of the proof is similar to what we did in Section \ref{sec3-2} for $d=1$. The market
considers a {\it particular} investor strategy of the form $f_m = f_m^* + \varepsilon f_m^\#$, and chooses the stock
movement $b=\pm 1$ based on the sign of $\hat{f} - f_m$, where $\hat{f}$ is the investor's choice. Roughly speaking,
the market uses his discretion over the stock price evolution to ``force'' the investor to use the particular strategy
(much as we did for $d=1$ in \eqref{market-strategy-case1}--\eqref{market-strategy-case2}). In doing so, the
market loses all control over the cycle decomposition of the resulting walk; therefore his worst-case
estimate involves the minimum rate at which regret accumulates (the minimization being over all simple cycles).
Our market's linear program, which chooses $f_m^\#$ to maximize this, seems quite different from the dual
of the investor's linear program.
\end{enumerate}

The rest of this subsection explains point (a). The investor's linear program chooses $f_m^\#$ to
$$
\min_{f_m^\#} \max_{\rm cycles} \ (\mbox{average rate at which regret accumulates}).
$$
Normalizing as we did in Section \ref{sec4-1}, and indexing the simple cycles by $i=1,\ldots,n$, this amounts to
\begin{equation} \label{investor-lp-as-minmax}
\min_{\beta_m} \max_{1 \leq i \leq n} \frac{1}{|s_i|} \sum_{
\substack{{\rm vertices} \, m \\ {\rm on} \, {\rm cycle} \, i}
} [(q(m) - r(m))^2 - b_{i,m} \beta_m]
\end{equation}
where $|s_i|$ is the number of edges in cycle $i$, and $b_{i,m} = +1$ if cycle $i$ leaves
vertex $m$ by the edge from $m$ to $m_+$, while $b_{i,m} = -1$ if cycle $i$
leaves vertex $m$ by the edge from $m$ to $m_-$. (Note that $b_{i,m} = a_{i,m+1}$, according to \eqref{first-cols-of-A}.)
The inner max is not changed if we permit ``mixed strategies,'' i.e. a probability distribution over the possible
cycles. So the investor's linear program solves
\begin{equation} \label{investor-lp-as-minmax}
\min_{\beta_m} \max_{\substack{p_1 + \cdots + p_n = 1\\ p_i \geq 0 }}
\sum_{i=1}^n p_i \frac{1}{|s_i|} \sum_{
\substack{{\rm vertices} \, m \\ {\rm on} \, {\rm cycle} \, i}
} [(q(m) - r(m))^2 - b_{i,m} \beta_m] .
\end{equation}
The dual is obtained by switching the min and the max, then evaluating the inner minimization. Since the
variables $\beta_m$ are unbounded,
\begin{multline} \label{dual-of-investor-lp}
\min_{\beta_m} \sum_{i=1}^n p_i \frac{1}{|s_i|} \sum_{
\substack{{\rm vertices} \, m \\ {\rm on} \, {\rm cycle} \, i}
} [(q(m) - r(m))^2 - b_{i,m} \beta_m] = \\
\left\{ \begin{array}{cl} -\infty & \mbox{if the variables $\beta_m$ don't cancel out}\\
\sum_{i=1}^n p_i \frac{1}{|s_i|} \sum_{
\substack{{\rm vertices} \, m \\ {\rm on} \, {\rm cycle} \, i}
} (q(m) - r(m))^2 & \mbox{if they do}.
  \end{array} \right.
\end{multline}
Thus: the dual of the investor's linear program involves mixed strategies over the simple cycles that make the market
insensitive to the investor's choice of $\beta_m$ (or equivalently, $f_m^\#$); the optimal mixed strategy is the one
that maximizes the value of \eqref{dual-of-investor-lp}. Since min-max equals max-min in this setting, the value
achieved by the optimal mixed strategy is the same as that of the investor's linear program.

Does the market have access to such mixed strategies? We suppose so, since the market can choose any walk it likes.

As noted earlier, throughout this section we have ignored the fact that $u_t$, $u_\eta$, and
$D = \half \langle D^2u \cdot \frac{\nabla^\bot u}{u_{\eta}}, \frac{\nabla^\bot u}{u_{\eta}}\rangle$ are not really
constant. Our upper and lower bounds for $d \geq 2$, presented in Section \ref{bounds-d>1}, deal with this
issue -- but only for linear programs discussed in Sections \ref{sec4-1} and \ref{sec4-2}. We have not attempted to
address this issue for the dual of the investor's linear program.

\subsection{Coalescence of the optimal values for $d\leq 4$} \label{sec4-4}

Recall from Corollary \ref{corollary:indifference-implies-equal-values} that the investor's linear program and the market's
linear program have the same optimal value if there is a choice of $ \{\beta_m \}_{m=0}^{2^d -1}$ that achieves indifference. We have
already shown the existence of such a choice when $d=1$. This subsection shows the existence of such a choice when
$d=2$, $3$, or $4$. Alas, since our analysis is by brute force, it offers little insight about whether indifference is
also achievable for $d \geq 5$.
\bigskip

\noindent {\sc The case $d=2$.} The graph for $d=2$ was shown in Figure \ref{fig:graph(d=2)withoutcosts}. It is easy to
see that there are six simple cycles, and we enumerated them in Section \ref{sec2-2}. Introducing the notation
$$
\gamma_m = (q(m)-r(m))^2,
$$
indifference requires that $\beta_0, \ldots, \beta_3$ satisfy
\begin{equation}
\begin{gathered}
\gamma_0 + \beta_0 = M\\
\gamma_3 - \beta_3 = M\\
\gamma_1 + \gamma_2 + \beta_1 - \beta_2 = 2M \label{eqns-for-indifference-d=2}\\
\gamma_0 + \gamma_1 + \gamma_2 - \beta_0 + \beta_1 + \beta_2 = 3M\\
\gamma_1 + \gamma_2 + \gamma_3 - \beta_1 - \beta_2 + \beta_3 = 3M\\
\gamma_0 + \gamma_1 + \gamma_2 + \gamma_3 - \beta_0 - \beta_1 + \beta_2 + \beta_3 = 4M
\end{gathered}
\end{equation}
and we know the value of $M$ from Lemma \ref{lemma:investor-LP}(c). Elementary manipulation
reveals that indifference is achieved when
\begin{align*}
& M = \frac{\gamma_0 + \gamma_1 + \gamma_2 + \gamma_3}{4}\\
& \beta_0=\beta_2 = M-\gamma_0\\
& \beta_1=\beta_3 = \gamma_3 - M.
\end{align*}

\noindent {\sc The case $d=3$.}
It is straightforward to draw the $d=3$ graph (see e.g. Figure 1 of \cite{Fredricksen}). It is possible (though laborious)
to check by hand that there are $19$ simple cycles, and to write down the $d=3$ analogue of \eqref{eqns-for-indifference-d=2}
(a system of $19$ linear equations in the $9$ unknowns $\gamma_0, \ldots, \gamma_7$ and $M$). Solving that system, one
finds that indifference is achieved for $d=3$ when
\begin{align*}
& M = \frac{\sum_{m=0}^{7} \gamma_i}{8}\\
& \beta_0=\beta_4 = M - \gamma_0\\
& \beta_1=\beta_5 = -M + \frac{-\gamma_2 + \gamma_3 + \gamma_6 + \gamma_7}{2}\\
& \beta_2=\beta_6 =M - \frac{\gamma_0 + \gamma_1 + \gamma_4 - \gamma_5}{2}\\
& \beta_3=\beta_7 = \gamma_7 - M.
\end{align*}

\noindent {\sc The case $d=4$.} The $d=4$ graph is not planar, but it is still not difficult to visualize
(see e.g. Figure 1 of \cite{Fredricksen}). We wrote a Matlab program to enumerate the simple cycles; it found $179$
of them. Using the results, we looked numerically for linear combinations of $\{ \gamma_m \}_{m=0}^{15}$ that
achieve indifference. This led to the conclusion that indifference is achieved for $d=4$ when
\begin{align*}
& M = \frac{\sum_{i=0}^{15}\gamma_i}{16}\\
& \beta_0=\beta_8 =M - \gamma_0\\
& \beta_1=\beta_9 = -M + \frac{-2\gamma_2 + 2\gamma_3 - \gamma_4 - \gamma_5 + \gamma_6 + \gamma_7 + \gamma_{12} +
\gamma_{13} + \gamma_{14} + \gamma_{15}}{4}\\
& \beta_2=\beta_{10} = M - \frac{\gamma_0 + \gamma_1 + \gamma_2 + \gamma_3 + 2\gamma_4 - 2\gamma_5
+ \gamma_8 + \gamma_9 - \gamma_{10} - \gamma_{11}}{4}\\
& \beta_3=\beta_{11} = -M + \frac{-2\gamma_6 + 2\gamma_7 + 2\gamma_{14} + 2\gamma_{15}}{4}\\
& \beta_4=\beta_{12} =M - \frac{2\gamma_0 + 2\gamma_1 + 2\gamma_8 - 2\gamma_9}{4}\\
& \beta_5=\beta_{13}= - M + \frac{-\gamma_4 - \gamma_5 + \gamma_6 + \gamma_7 - 2\gamma_{10} +
2\gamma_{11} + \gamma_{12} + \gamma_{13} + \gamma_{14} + \gamma_{15}}{4}\\
& \beta_6=\beta_{14} =M - \frac{\gamma_0 + \gamma_1 + \gamma_2 + \gamma_3 + \gamma_8 + \gamma_9 -
\gamma_{10} - \gamma_{11} + 2\gamma_{12} - 2\gamma_{13}}{4}\\
& \beta_7=\beta_{15} = - M + g_{15}.
\end{align*}

We wonder whether indifference might be achievable for any $d$. Alas, the brute force approach we have
used for $d \leq 4$ does not provide much guidance. (It does, however, provide some hints; in particular, the choices
achieving indifference for $d=2, 3, 4$ all have the symmetry that $\beta_m = \beta_{m + 2^{d-1}}$ for
$0 \leq m \leq 2^{d-1} - 1$. This means that $\beta_m$ actually depends only on the most recent $d-1$ market moves.)

\section{Upper and lower bounds for general $d$} \label{bounds-d>1}

This section provides our fully rigorous upper and lower bounds for the general case, when our two experts use $d$ days of
data. The arguments are in many ways parallel to those presented in Section \ref{bounds-d=1} where we considered the case $d=1$.
The main differences are:
\begin{enumerate}
\item[(i)] For $d=1$ there is a choice of $\beta_m$ that achieves indifference (in the sense of Lemmas \ref{lemma:investor-LP}(c)
and \ref{lemma:market-LP}(c)), and the proofs of both the upper and lower bounds used the associated value of $f_m^\#$ (this was
the logic behind \eqref{f-when-d=1}). For general $d$, our upper bound uses the investor's linear program to
determine $f_m^\#$, whereas our lower bound uses the market's linear program.

\item[(ii)] In formulating the linear programs at the beginning of Section \ref{LP} we ignored the fact that $u_t$, $u_\eta$, and
$D = \half \langle D^2u \cdot \frac{\nabla^\bot u}{u_{\eta}}, \frac{\nabla^\bot u}{u_{\eta}}\rangle$ are functions of
$x$ and $t$. For $d=1$, the errors associated with their variability were relatively easy to control
(see \eqref{same-number}--\eqref{wandering-of-D}). While that argument did not make explicit use of the graph,
it basically relied on the simple cycle structure of the $d=1$ graph. The corresponding arguments for general $d$ are more
complicated since the cycle structure of the graph is less controlled.
\end{enumerate}

\noindent In connection with the second point, we shall need the following refinement of Lemma \ref{walkcycles}.

\begin{lemma} \label{walkcycles-refined}
When a closed walk on a directed graph is decomposed into a union of simple cycles by the argument
used to prove Lemma \ref{walkcycles}, the decomposition has the following property: for any simple cycle $s$
that appears in the decomposition more than once, each instance of $s$ is completed before the next
begins (as one traverses the walk from beginning to end).
\end{lemma}
\begin{proof}
Fixing the graph under consideration, we use complete induction on the length of the walk.
The shortest possible length of a closed walk is the length of the
shortest simple cycle; for walks of this length the result is trivial.

For the inductive step, we must show that if the result is true for closed walks of length up to $N-1$ then it is
also true for closed walks of length $N$. So consider a closed walk of length $N$ whose vertices (in order) are
$a_0, a_1, \ldots , a_N$ with $a_N = a_0$. Consider (as in the proof of Lemma \ref{walkcycles}) the beginning of the walk, up
to the first time a vertex is repeated:
$$
a_0  \ldots a_i  \ldots a_j \ldots a_N \quad \mbox{where $a_j = a_i$ is the first repetition}.
$$
Then $s = a_i  \ldots a_j$ is a simple cycle, and it is the first cycle in the decomposition of the walk. The
rest of the decomposition is obtained by considering the walk that remains after removal of this cycle, namely
$a_0 \ldots a_{i-1} , a_j , a_{j+1} \ldots a_N$ and applying the same argument (repeatedly). The inductive hypothesis applies
to this shortened walk.

For simple cycles other than $s$, the inductive hypotheses assures us that each instance is complete before the next begins,
as one traverses the shortened walk. Therefore the same is true of the original walk.

As for the simple cycle $s$: since $a_i = a_j$ was the first repeated node, none of the vertices $a_0 \ldots a_{i-1}$ participate
in $s$. Therefore in the decomposition of the shortened walk, no instances of $s$ are begun during the initial segment
$a_0 \ldots a_{i-1}$. So the first instance of $s$ is completed before the others begin; and by the inductive hypothesis, each
other instance of $s$ is completed before the next begins.
\end{proof}

\subsection{The upper bound, when $\varphi$ is regular} \label{sec5-1}
Let $u^\varepsilon(t,m,\xi,\eta)$ be
defined by the dynamic programming principle \eqref{DPP} with final value
$u^\varepsilon (T,m,\xi,\eta) = \varphi(\xi,\eta)$ (it is defined only for times $t$ such
that $(T-t)/\varepsilon^2$ is an integer); and let $u(t,\xi, \eta)$ be the solution of the PDE
\begin{align} \label{PDEsharpDU}
\begin{split}
&u_t+\frac{1}{2} C^\#_d\langle D^2 u \frac{\nabla^\bot u}{ u_{\eta}},
\frac{\nabla^\bot u}{ u_{\eta}} \rangle = 0,  \\
& u(T,\xi,\eta) = \varphi(\xi, \eta),
\end{split}
\end{align}
where $C^\#_d$ is the optimal value of the investor's linear program \eqref{LP1m}.
Our goal is to prove
\begin{theorem} \label{Dup}
Let $d\geq 1$ be an integer and assume the solution $u$ of \eqref{PDEsharpDU} satisfies
\eqref{require-u-smooth}--\eqref{require-time-deriv-neg}.
Then there is a constant $C$ (independent of $\varepsilon$, $t$, and $T$) such that
\begin{equation} \label{upper-bound-statement-d}
u^\varepsilon(t, m, \xi, \eta)\leq u(t,\xi,\eta)+C[(T-t) + \varepsilon]\varepsilon
\quad \mbox{for $t < T$, $\xi \in \mathbb{R}$, $\eta \in \mathbb{R}$, and $m\in\{0,1\}^d$}
\end{equation}
whenever $\varepsilon$ is small enough and $t$ is such that $N=(T-t)/\varepsilon^2$ is an integer.
\end{theorem}

\begin{proof}
Our overall strategy is the same as for Theorem \ref{1up}:
to estimate $u^\varepsilon(t_0,m_0,\xi_0,\eta_0)$, we shall
define a sequence $(t_k,m_k,\xi_k, \eta_k)$ along which $u^\varepsilon$ is monotone
\begin{equation} \label{chain10}
u^\varepsilon(t_0, m_0, \xi_0, \eta_0) \leq  u^\varepsilon(t_1, m_1, \xi_1, \eta_1)\leq
\dots \leq u^\varepsilon(t_N, m_N, \xi_N, \eta_N)
\end{equation}
with $t_N=T$, so that
\begin{equation} \label{chain11}
u^\varepsilon(t_N, m_N, \xi_N, \eta_N)=\varphi(\xi_N, \eta_N)=u(t_N, \xi_N, \eta_N).
\end{equation}
Then we'll show that
\begin{equation} \label{chain12}
u(t_N,\xi_N, \eta_N) - u(t_0, \xi_0,\eta_0) \leq C[ (T-t_0) + \varepsilon ] \varepsilon.
\end{equation}
These estimates lead immediately to \eqref{upper-bound-statement-d}.

Our choice of $\{(t_k,m_k,\xi_k,\eta_k)\}_{k=1}^N$ is entirely parallel to what was done for
Theorem \ref{1up}. It suffices to explain the choice of $(t_1,m_1,\xi_1,\eta_1)$ (then the rest
of the sequence is determined similarly, step by step). Recall that the dynamic programming
principle was written compactly in equations \eqref{defn-v}--\eqref{DPP-compact}, and it gave
\begin{equation} \label{DPP-after-choice-of-f-bis}
u^\varepsilon(t_0, m_0, \xi_0,\eta_0) \leq \max_{b_0=\pm 1}
u^\varepsilon(t_0+\varepsilon^2, m_{b_0}, \xi_0 +\varepsilon b_0 v^1, \eta_0 + \varepsilon b_0 v^2)
\end{equation}
where $(v^1, v^2)$ are determined by the investor's choice of $f$ via \eqref{defn-v}. Our choice of
$f = f_{t_0, m_0}^* + \varepsilon f_{t_0,m_0}^\#$ is guided by the heuristic discussion in
Section \ref{sec2-4} (which determines $f_{t_0, m_0}^*$) combined with the investor's linear program
(which determines $f_{t_0,m_0}^\#$ via \eqref{f_m^sharp-vs-beta_m}):
\begin{equation} \label{f-when-d}
\begin{split}
f_{t_0,m_0} &= f^*_{t_0,m} + \varepsilon f^\#_{t_0,m_0}, \ \mbox{with}\\
f^*_{t_0,m_0} & =  \frac{(q(m_0)-r(m_0))u_{\xi}+(q(m_0)+r(m_0)) u_{\eta}}{2u_{\eta}} \ \mbox{and} \\
f^\#_{t_0,m_0} & = \frac{D}{2 u_\eta} \beta_{m_0}.
\end{split}
\end{equation}
Here $u_\xi$, $u_\eta$, and
$D = \half \langle D^2u \cdot \frac{\nabla^\bot u}{u_{\eta}}, \frac{\nabla^\bot u}{u_{\eta}}\rangle$ are evaluated
at $(t_0, \xi_0, \eta_0)$, while $\beta_{m_0}$ comes from the solution of the investor's linear program \eqref{LP1m}.
(If the linear program has more than one solution, we choose one and use it throughout the proof.) The proper choice
of $(t_1,m_1,\xi_1,\eta_1)$ is now clear: taking $b_0$ to achieve the max on the RHS of
\eqref{DPP-after-choice-of-f-bis}, the choice $t_1 = t_0 + \varepsilon^2$, $m_1=(m_0)_{b_0}$,
$\xi_1= \xi_0 + \varepsilon b_0 v^1$, $\eta_1 = \eta_0 + \varepsilon b_0 v^2$ satisfies the desired inequality
$u^\varepsilon(t_0, m_0, \xi_0,\eta_0) \leq u^\varepsilon(t_1, m_1, \xi_1,\eta_1)$. The rest of the sequence
$(t_k, m_k, \xi_k, \eta_k)$ is determined similarly, up to $k=N$, when $t_N = T$.

We now depart a bit from the proof of Theorem \ref{1up}. The sequence $m_0, m_1, \ldots, m_N$ is a walk
on our directed graph, but it is not in general closed (since we cannot expect that $m_N = m_0$). It can, however,
be extended to a closed walk by adding some steps at the end. Indeed, as already noted in
Section \ref{LP} our graph is Eulerian, i.e. there is a closed walk (an Eulerian circuit) that traverses
each edge exactly once; the  desired extension can be achieved by adding part of an Eulerian circuit,
starting from $m_N$ and stopping upon arrival at $m_0$. We let $N'$ be the final index of the extended walk
(so $m_{N'} = m_0$); note that the number of extra steps is bounded by the total number of edges in the graph:
\begin{equation} \label{controlled-number-of-extra-steps}
N' - N \leq 2^{d+1}.
\end{equation}
The expression $u^\varepsilon(t_k, m_k, \xi_k,\eta_k)$ is undefined for $k>N$. However it is convenient to introduce
a suitable choice of $(t_k, \xi_k, \eta_k)$ for $k>N$, so that $u(t_k, \xi_k, \eta_k)$ will make sense.
We take the time increments to be trivial: $t_k = t_N = T$ for $k>N$. For the spatial increments we use
the same definition as for $k < N$: defining $b_k = \pm 1$ for $k \geq N$ by $(m_k)_{b_k} = m_{k+1}$, we set
$$
\xi_{k+1}= \xi_k + \varepsilon b_k v^1, \quad \eta_{k+1} = \eta_k + \varepsilon b_k v^2
$$
for $k \geq N$, where $(v^1, v^2)$ are determined by the investor's choice of $f$ via \eqref{defn-v}, the choice of $f$
still being given by \eqref{f-when-d} (with $t_0$ and $m_0$ replaced by $t_k $ and $m_k$, and with
$u_\xi$, $u_\eta$, and $D = \half \langle D^2u \cdot \frac{\nabla^\bot u}{u_{\eta}}, \frac{\nabla^\bot u}{u_{\eta}}\rangle$
evaluated at $(t_k, \xi_k, \eta_k)$). This choice has the key feature that when the increment
$$
U_k = u(t_{k+1},\xi_{k+1},\eta_{k+1}) - u(t_k,\xi_k,\eta_k)
$$
is estimated for $k \geq N$ by Taylor expansion, there is no $u_t$ term (since $t_{k+1}=t_k$) and
the first-order terms involving the increments of $\xi$ and $\eta$ vanish
(see \eqref{TaylorMinMax}--\eqref{value-of-B}); as a result, we have
\begin{equation} \label{increments-beyond-N}
|U_k| \leq C \varepsilon^2 \quad \mbox{for $k \geq N$}.
\end{equation}
As a reminder, we also have a convenient estimate for the increments at $k < N$:
\begin{equation} \label{increments-before-N}
U_k = \varepsilon^2 L(t_k, m_k, b_k, \xi_k, \eta_k, f^\#_{t_k,m_k}) + O(\varepsilon^3)
\end{equation}
where $L$ is defined by \eqref{defn-of-L}.

We now start the proof of \eqref{chain12}. In view of \eqref{controlled-number-of-extra-steps} and
\eqref{increments-beyond-N}, it suffices to prove that
$$
u(t_{N'},\xi_{N'}, \eta_{N'}) - u(t_0, \xi_0,\eta_0) \leq C[ (T-t_0) + \varepsilon ] \varepsilon.
$$
Since the walk $m_0, \ldots, m_{N'}$ is closed, it is a union of simple cycles by Lemma \ref{walkcycles}.
Now,
\begin{equation} \label{sum-of-all-increments}
u(t_{N'},\xi_{N'}, \eta_{N'}) - u(t_0, \xi_0,\eta_0) = \sum_{k=0}^{N'-1} U_k,
\end{equation}
and each increment is associated with a step of our closed walk, so RHS of \eqref{sum-of-all-increments} can be
reorganized using the walk's cycle decomposition.

The total contribution of all cycles that involve the extended part of the walk (i.e., cycles that
include a node $m_k$ with $k > N$) is of order $\varepsilon^2$, since the number of such cycles is finite
(with an estimate depending only on $d$) and each instance of any cycle makes a contribution of order
$\varepsilon^2$. We shall refer to the cycles that don't involve the extended part of the walk as
the {\it remaining cycles}.

The $O(\varepsilon^3)$ error terms in \eqref{increments-before-N} don't bother us, since
they accumulate to an error of at most $C (T-t) \varepsilon$. Thus to prove \eqref{chain12} we need only show that
\begin{equation} \label{essential-task-ub}
\begin{split}
&\mbox{the terms $\varepsilon^2 L(t_k, m_k,b_k,\xi_k,\eta_k,f^\#_{t_k,m_k})$ coming from the}\\[-1ex]
&\mbox{remaining cycles sum to at most $C(T-t) \varepsilon$.}
\end{split}
\end{equation}
Let $\{ s_i \}_{i=1}^n$ be a list of the simple cycles on the graph; we shall (as usual) write $|s_i|$ for the number of edges
in $s_i$. Suppose that among the remaining cycles, $s_i$ appears $\sigma_i$ times.
We sum $L(t_k, m_k,b_k,\xi_k, \eta_k,f^\#_{t_k,m_k})$ over the remaining cycles in stages: first
over the edges of an instance $\alpha$ of a cycle $s_i$, then over the $\sigma_i$ instances of this cycle,
then over all the distinct cycles $s_1, \dots, s_n$. Thus, the quantity to be estimated is
\begin{equation} \label{Lsum}
\varepsilon^2 \sum_{\substack{{\rm remaining}\\{\rm cycles}}} L(t_k, m_k,b_k,\xi_k, \eta_k,f^\#_{t_k,m_k})
= \varepsilon^2 \sum_{i=1}^{n}\sum_{\alpha=1}^{\sigma_i}\sum_{j=1}^{|s_i|}L (t^\alpha_{i,j}, m^\alpha_{i,j},b^\alpha_{i,j},\xi_{i,j}^\alpha,\eta_{i,j}^\alpha,f^{\#}_{t^\alpha_{i,j},m^ \alpha_{i,j}}).
\end{equation}
Here the time steps have been relabeled using the cycle decomposition: $t^\alpha_{i,j}$ is the time when
instance $\alpha$ of cycle $s_i$ takes its $j$th step, $j=1,\ldots,|s_i|$. Note that by Lemma
\ref{walkcycles-refined}, for each $i$ and each $\alpha < \sigma_i$,
\begin{equation} \label{ordering-of-times}
t^\alpha_{i,1} < t^\alpha_{i,2} < \ldots < t^\alpha_{i,|s_i|} <
t^{\alpha +1}_{i,1} < t^{\alpha +1}_{i,2} < \ldots < t^{\alpha + 1}_{i,|s_i|}.
\end{equation}

In formulating our linear program we treated $u_t$, $u_\eta$, etc. as being constant, whereas in fact
they vary with space and time. To deal with this, we need to compare
$L(t^\alpha_{i,j}, m^\alpha_{i,j},b^\alpha_{i,j},\xi_{i,j}^\alpha,
\eta_{i,j}^\alpha,f^{\#}_{t^\alpha_{i,j},m^ \alpha_{i,j}})$
to the analogous quantity
$L(t^\alpha_{i,1}, m^\alpha_{i,j},b^\alpha_{i,j},\xi_{i,1}^\alpha,
\eta_{i,1}^\alpha,f^{\#}_{t^\alpha_{i,1},m^\alpha_{i,j}})$
obtained by freezing $u_t$, $u_\eta$, etc to their
values at time $t^\alpha_{i,1}$ when the $\alpha$th instance of the $i$th cycle begins. Combining
the definition \eqref{defn-of-L} of $L$ with the choice \eqref{f-when-d} of $f_{t,m}^\#$, we have
$$
L(t,m,b,\xi,\eta,f_{t,m}^\#) = u_t + (q(m)-r(m))^2 D - b \beta_m D,
$$
in which $u_t$ and
$D = \half \langle D^2u \cdot \frac{\nabla^\bot u}{u_{\eta}}, \frac{\nabla^\bot u}{u_{\eta}}\rangle$
are evaluated at $(t,\xi, \eta)$. Our hypotheses \eqref{require-u-smooth}--\eqref{require-eta-deriv-pos}
assure that $u_t$ and $D$ are uniformly Lipschitz continuous functions of $t$, $\xi$, and $\eta$. Since
the number of time steps from $t^\alpha_{i,1}$ to $t^\alpha_{i,j}$ is
$(t^\alpha_{i,j} - t^\alpha_{i,1})/\varepsilon^2$ and the change in $(t,\xi,\eta)$ is of order $\varepsilon$
at each time step,
\begin{equation} \label{Lo}
L(t^\alpha_{i,j}, m^\alpha_{i,j},b^\alpha_{i,j},\xi_{i,j}^\alpha,\eta_{i,j}^\alpha,f^{\#}_{t^ \alpha_{i,j},m^ \alpha_{i,j}})
=L(t^\alpha_{i,1}, m^\alpha_{i,j},b^\alpha_{i,j},\xi_{i,1}^\alpha,\eta_{i,1}^\alpha,f^{\#}_{t^ \alpha_{i,1},m^\alpha_{i,j}}) +
\frac{t^\alpha_{i,j}-t^\alpha_{i,1}}{\varepsilon^2}O(\varepsilon).
\end{equation}
Focusing on the terms associated with a particular instance of a particular cycle (thus, fixing $\alpha$ and $i$), we conclude
that
\begin{multline} \label{chain8}
\sum_{j=1}^{|s_i|}
L(t^\alpha_{i,j}, m^\alpha_{i,j},b^\alpha_{i,j},\xi_{i,j}^\alpha,\eta_{i,j}^\alpha,f^{\#}_{t^ \alpha_{i,j},m^ \alpha_{i,j}})\\
= \sum_{j=1}^{|s_i|}
L(t^\alpha_{i,1}, m^\alpha_{i,j},b^\alpha_{i,j},\xi_{i,1}^\alpha,\eta_{i,1}^\alpha,f^{\#}_{t^ \alpha_{i,1},m^ \alpha_{i,j}})
+ \sum_{j=1}^{|s_i|}\varepsilon^{-2}(t^\alpha_{i,j}-t^\alpha_{i,1}) O(\varepsilon)\\
=|s_i|\left[
u_t + D^\alpha_{i,1} \sum_{j=1}^{|s_i|} \frac{1}{|s_i|}
\left[(q(m^\alpha_{i,j})-r(m^\alpha_{i,j}))^2 -b^\alpha_{i,j}\beta_{m^\alpha_{i,j}} \right]
\right] +
\sum_{j=1}^{|s_i|} (t^\alpha_{i,j}-t^\alpha_{i,1}) O( \varepsilon^{-1})
\end{multline}
where $D^\alpha_{i,1}$ is the value of
$\frac{1}{2}\langle D^2u \cdot \frac{\nabla^\bot u}{u_{\eta}}, \frac{\nabla^\bot u}{u_{\eta}}\rangle)$
evaluated at $(t^\alpha_{i,1},\xi_{i,1}^\alpha,\eta_{i,1}^\alpha)$, and $u_t$ is also evaluated at this location.
We come now to the {\it key point}:
\begin{equation} \label{key-point-ub}
\sum_{j=1}^{|s_i|}\frac{1}{|s_i|}
\left[(q(m^\alpha_{i,j})-r(m^\alpha_{i,j}))^2 -b^\alpha_{i,j}\beta_{m^\alpha_{i,j}} \right] \leq C_d^\# ,
\end{equation}
{\it since} $(\beta^0, \ldots, \beta_{2^d -1}, C_d^\#)$ {\it is a feasible point for the investor's linear program.}
Our hypothesis that $u_t \leq 0$ and the PDE \eqref{PDEsharpDU} assure that $D^\alpha_{i,1} \geq 0$; so
\eqref{chain8} is bounded above by
\begin{equation} \label{key-point-ub-bis}
|s_i|[u_t + \frac{C^\#_d}{2}\langle D^2 u \frac{\nabla^\bot u}{u_{\eta}},
\frac{\nabla^\bot u}{u_{\eta}}\rangle]+ \sum_{j=1}^{|s_i|} (t^\alpha_{i,j}-t^\alpha_{i,1})O(\varepsilon^{-1}).
\end{equation}
Since the first term vanishes by \eqref{PDEsharpDU}, we have shown that
\begin{equation} \label{chain8a}
\sum_{j=1}^{|s_i|}
L(t^\alpha_{i,j}, m^\alpha_{i,j},b^\alpha_{i,j},\xi_{i,j}^\alpha,\eta_{i,j}^\alpha,f^{\#}_{t^ \alpha_{i,j},m^ \alpha_{i,j}}) \leq
\sum_{j=1}^{|s_i|} (t^\alpha_{i,j}-t^\alpha_{i,1})O(\varepsilon^{-1})
\end{equation}
for every $i$ and $\alpha$.

To finish, we now sum over $i$ and $\alpha$ and change the order of summation:
$$
\mbox{value of \eqref{Lsum}} \leq \varepsilon^2\sum_{i=1}^{n}\sum_{\alpha=1}^{\sigma_i}\sum_{j=1}^{|s_i|}(t^\alpha_{i,j}-t^\alpha_{i,1})O(\varepsilon^{-1})
= \sum_{i=1}^{n} \sum_{j=1}^{|s_i|} \sum_{\alpha=1}^{\sigma_i}(t^\alpha_{i,j}-t^\alpha_{i,1})O(\varepsilon).
$$
By \eqref{ordering-of-times} this is at most
$$
\sum_{i=1}^{n} \sum_{j=1}^{|s_i|}[T-t_0] O(\varepsilon).
$$
Since $n$ (the number of cycles) and $\max_i |s_i|$ (the maximum length of a cycle) are constants (depending only on $d$), we
conclude as desired that
$$
\mbox{value of \eqref{Lsum}} \leq C(T-t_0) \varepsilon ,
$$
where $C$ depends only on $d$ and the constants implicit in our hypotheses on $u$, \eqref{require-u-smooth}--\eqref{require-time-deriv-neg}.
\end{proof}

\begin{remark} \label{lb-proof-uses-only-feasibility}
The only property of $C^\#_d$ that we used in the proof was the
existence of $\{\beta_m \}$ such that $(\beta_0, \ldots, \beta_{2^d-1},C_d^\#)$ is a feasible point for the
investor's linear program (see \eqref{key-point-ub}). As a reminder, this means that for each cycle $s$ on the graph,
$$
\frac{1}{|s|} \sum_{
\substack{{\rm vertices} \, m \\ {\rm on} \, {\rm cycle} \, s}
} [(q(m) - r(m))^2 - b_{i,m} \beta_m] \leq C_d^\#.
$$
where $|s|$ is the number of edges in $s$ and $b_{i,m}$ was defined after \eqref{investor-lp-as-minmax}.
By choosing $C^\#_d$ to be the optimal value of the investor's linear program we make $C^\#_d$ as small as possible,
optimizing the resulting bound. But since our argument uses only feasibility (not optimality), {\it any} feasible
point for the investor's linear program determines an upper bound on $u^\varepsilon$. For example, to get an upper
bound that's worse than that of Theorem \ref{Dup} but somewhat more explicit, one can take $\beta_m = 0$
for all $m$ and replace $C^\#_d$ in \eqref{PDEsharpDU} by
$$
\max_{s \in \{ \rm cycles \}} \frac{1}{|s|} \sum_{
\substack{{\rm vertices} \, m \\ {\rm on} \, {\rm cycle} \, s}
} (q(m) - r(m))^2 .
$$
\end{remark}

\begin{remark} \label{indifference-implies-sign-condition-not-needed}
The hypotheses of Theorem \ref{Dup} include that $u_t \leq 0$, an assumption we didn't need when $d=1$. It was needed for the
passage from \eqref{key-point-ub} to \eqref{key-point-ub-bis}. If there is a strategy that achieves indifference (in the sense
of Lemmas \ref{lemma:investor-LP}(c) and \ref{lemma:market-LP}(c)) then the inequality in \eqref{key-point-ub} becomes an
equality and the sign of $u_t$ becomes irrelevant. As a reminder: we showed in Section \ref{sec4-4} that indifference is
achievable for $d \leq 4$.
\end{remark}

\subsection{The lower bound, when $\varphi$ is regular} \label{sec5-2}

For our lower bound on $u^\varepsilon$, the function $u$ solves a PDE similar to that used for the lower bound, but with
a different ``diffusion constant'': throughout this subsection, $u(t,\xi, \eta)$ is the solution of the PDE

\begin{align} \label{PDEsharpDL}
\begin{split}
&u_t+\frac{1}{2} C^*_d\langle D^2 u \frac{\nabla^\bot u}{ u_{\eta}},
\frac{\nabla^\bot u}{ u_{\eta}} \rangle = 0,  \\
& u(T,\xi,\eta) = \varphi(\xi, \eta),
\end{split}
\end{align}
where $C^*_d$ is the optimal value of the market's linear program \eqref{LP2m}. Our goal is to prove

\begin{theorem} \label{Dlow}
Let $d\geq1$ be an integer, and assume the solution $u$ of \eqref{PDEsharpDL} satisfies
\eqref{require-u-smooth}--\eqref{require-time-deriv-neg}. Then there is a constant $C$ (independent
of $\varepsilon$, $t$, and $T$) such that
\begin{equation} \label{lower-bound-statement-d}
u^\varepsilon(t, m, \xi, \eta)\geq u(t,\xi,\eta)- C[(T-t) + \varepsilon]\varepsilon
\quad \mbox{for $t < T$, $\xi \in \mathbb{R}$, $\eta \in \mathbb{R}$, and $m\in\{0,1\}$}
\end{equation}
whenever $\varepsilon$ is small enough and $t$ is such that $N=(T-t)/\varepsilon^2$ is an integer.
\end{theorem}

\begin{proof}
No new ideas are needed, beyond those introduced for the lower bound when $d=1$ (Theorem \ref{1low}) and the upper
bound for general $d$ (Theorem \ref{Dup}). Therefore we will be somewhat brief.

Our lower bound for $d=1$ was associated with a specific strategy for the market, which was summarized at the beginning
of Section \ref{sec3-2}. Our lower bound for general $d$ uses an analogous strategy. Briefly: the market identifies a specific
investor strategy of the form $f_m = f_m^* + \varepsilon f_m^\#$ by solving the market's linear program, and chooses the market's
movements to penalize the investor if she doesn't make that choice. In more detail: the market's strategy is organized around
the analogue of \eqref{f-when-d} obtained using the market's linear program rather than the investor's:
\begin{equation} \label{f-when-d-lb}
\begin{split}
f_{t,m} &= f^*_{t,m} + \varepsilon f^\#_{t,m}, \ \mbox{with}\\
f^*_{t,m} & =  \frac{(q(m)-r(m))u_{\xi}+(q(m)+r(m)) u_{\eta}}{2u_{\eta}} \ \mbox{and} \\
f^\#_{t,m} & = \frac{D}{2 u_\eta} \beta_{m}.
\end{split}
\end{equation}
where $\{\beta_{m} \}_{m=0}^{2^d -1}$ come from the solution of the market's linear program \eqref{LP1m}.
There are two cases, identical to \eqref{market-strategy-case1} and \eqref{market-strategy-case2}:
\bigskip

\noindent {\bf Case 1}: If the investor's choice doesn't nearly zero out the ``first-order term'' in the
Taylor-expansion-based estimate of the increment of $u$, then the market chooses $b$ to make that term positive;
quantitatively,
$$
\mbox{if the investor's choice $f$ has $|f-f_m^*| \geq \gamma \varepsilon$ then the market chooses $b$
so that $-b (f-f_m^*) \geq 0$.}
$$

\noindent {\bf Case 2:} When case 1 doesn't apply, it is convenient to express the
investor's choice $f$ as $f=f_m^* + \varepsilon f_m^\# + \varepsilon X$, where $f_m^*$ and $f_m^\#$ are
defined by \eqref{f-when-d-lb} (this relation defines $X$). The investor is more
optimistic than our conjectured optimal strategy if $X>0$, and more pessimistic if $X<0$.
In the former case the market makes the stock go down, and in the latter case it makes the stock
go up -- in each case giving the investor an unwelcome surprise; quantitatively:
$$
\mbox{if the investor's choice $f$ has $|f-f_m^*| < \gamma \varepsilon$ then the market chooses $b$
so that $-bX \geq 0$.}
$$

Under conditions on $\gamma$ analogous to \eqref{gamma-condition1} and \eqref{gamma-condition2}, an argument
entirely parallel to that of Theorem \ref{1low} produces a sequence $(t_k,m_k,\xi_k, \eta_k)$, starting from any given
$(t_0,m_0,\xi_0,\eta_0)$ and ending when $k=N=(T-t_0)/\varepsilon^2$ so that $t_N = T$, such that
\begin{equation} \label{chain10-lb}
u^\varepsilon(t_0, m_0, \xi_0, \eta_0) \geq  u^\varepsilon(t_1, m_1, \xi_1, \eta_1)\geq
\dots \geq u^\varepsilon(t_N, m_N, \xi_N, \eta_N)
\end{equation}
for which the increments of $u$ satisfy the analogue of \eqref{like-ub}:
\begin{equation} \label{inequalities-for-increments}
U_k = u(t_{k+1}, \xi_{k+1}, \eta_{k+1}) - u(t_k,\xi_k,\eta_k) \geq
\varepsilon^2 L(t_k,m_k,b_k,\xi_k,\eta_k,f_{t_k,m_k}^\#) + O(\varepsilon^3)
\end{equation}
Then arguments entirely parallel to the those used for Theorem \ref{Dup} show that
\begin{equation} \label{chain12-lb}
u(t_{N},\xi_{N}, \eta_{N}) - u(t_0, \xi_0,\eta_0) \geq - C[ (T-t_0) + \varepsilon ] \varepsilon.
\end{equation}
The analogue of \eqref{key-point-ub} in this setting is of course
\begin{equation} \label{key-point-lb}
\sum_{j=1}^{|s_i|}\frac{1}{|s_i|}
\left[(q(m^\alpha_{i,j})-r(m^\alpha_{i,j}))^2 -b^\alpha_{i,j}\beta_{m^\alpha_{i,j}} \right] \geq C_d^* ,
\end{equation}
which holds since $(\beta^0, \ldots, \beta_{2^d -1}, C_d^*)$ is a feasible point for the market's linear program.
Combining \eqref{chain10-lb}--\eqref{chain12-lb} with the fact that $u^\varepsilon(t_N, m_N, \xi_N, \eta_N) =
\varphi(t_N,\xi_N,\eta_N) = u(t_N, \xi_N, \eta_N)$ gives the desired lower bound
$$
u^\varepsilon(t_0, m_0, \xi_0, \eta_0 )\geq u(t_0,\xi_0,\eta_0)- C[(T-t) + \varepsilon]\varepsilon .
$$
\end{proof}

\begin{remark} \label{remark-after-lb}
The observations in Remarks \ref{lb-proof-uses-only-feasibility} and \ref{indifference-implies-sign-condition-not-needed} apply
here as well: since the proof of the lower bound uses only the feasibility (not the optimality) of
$(\beta_0, \ldots, \beta_{2^d-1},C_d^*)$ for the market's linear program, the same argument can be applied using
{\it any} feasible point for that linear program. For example, to get an upper
bound that's worse than that of Theorem \ref{Dlow} but more explicit, one can take $\beta_m = 0$
for all $m$ and replace $C_d^*$ in \eqref{PDEsharpDL} by
$$
\min_{s \in \{ \rm cycles \}} \frac{1}{|s|} \sum_{
\substack{{\rm vertices} \, m \\ {\rm on} \, {\rm cycle} \, s}
} (q(m) - r(m))^2 .
$$
\end{remark}

\begin{remark} \label{remark-on-matching-bounds} We know from Lemmas \ref{lemma:investor-LP} and \ref{lemma:market-LP} that
$C_d^* \leq C_d^\#$. Our upper and lower bounds match asymptotically in the limit $\varepsilon \rightarrow 0$ if and only if
$C_d^* = C_d^\#$. By Lemmas \ref{lemma:investor-LP}(c) and \ref{lemma:market-LP}(c), this relation holds when there is a strategy that
achieves indifference.
\end{remark}

\subsection{The classic case, when $\varphi= \frac{1}{2}(\eta + |\xi|)$} \label{sec5-3}

The classic goal of minimizing regret with respect to the best-performing expert corresponds to
using the non-smooth final-time data $\varphi = (\eta + |\xi|)/2$. In Section \ref{bounds-d=1}, which focused on the case
$d=1$, we showed in Theorem \ref{1classic} how our upper and lower bounds can be adapted to this case. The proof involved
approximating $\varphi$ by something a bit smoother, then examining the various error terms.

For the general case $d \geq 1$, the proofs of our upper and lower bounds (Theorems \ref{1up} and \ref{1low}) were largely
parallel to the case $d=1$. The main difference was in some sense bookkeeping, namely our use of the cycle decomposition of
a closed walk on the relevant graph; for $d=1$ the cycle decomposition was so simple that we avoided discussing it
explicitly, though it was implicit in our discussion of \eqref{same-number}--\eqref{wandering-of-D}.

In view of the parallels between our bounds for $d =1$ and those for the general case $d \geq 1$, it is not surprising that
the results in this section can be extended to the classic case $\varphi = (\eta + |\xi|)/2$.

\begin{theorem} \label{DclassicU}
Let $d\geq1$ be an integer and consider the classic final-time data $\varphi = (\eta + |\xi|)/2$. We consider, as usual,
the function $u^\varepsilon(t,m,\xi,\eta)$ defined by the dynamic programming principle \eqref{DPP}
with final-time data $\varphi$. Let $u_+(t,\xi,\eta)$ solve our upper-bound PDE \eqref{PDEsharpDU} and let
$u_-(t,\xi,\eta)$ solve our lower-bound PDE \eqref{PDEsharpDL}, with final-time data $\varphi$ in both cases. Then
there is a constant $C$ (independent of $\varepsilon$, $t$, and $T$) such that
\begin{equation} \label{estimate-classic-d}
u_-(t,\xi,\eta) - C \varepsilon |\log \varepsilon| \leq
 u^\varepsilon(t, m, \xi, \eta) \leq
 u_+(t,\xi,\eta) + C \varepsilon |\log \varepsilon|
\end{equation}
for all $t < T$, $\xi \in \mathbb{R}$, $\eta \in \mathbb{R}$, and $m\in\{0,1\}^d$, provided that
$\varepsilon$ is small enough and $t$ is such that $N=(T-t)/\varepsilon^2$ is an integer.
\end{theorem}

\begin{proof}
The proof requires no new ideas: one simply reviews the proofs of Theorems \ref{Dup} and \ref{Dlow}, estimating the magnitude of each
error term as in the proof of Theorem \ref{1classic}. We leave the details to the reader.
\end{proof}

\section{Appendix: solving our PDE using the linear heat equation} \label{APP}

Our bounds involve solutions of the final-value problem
\begin{align} \label{PDE-with-diffusion-const-C}
\begin{split}
&u_t+\frac{1}{2} C \langle D^2 u \frac{\nabla^\bot u}{ u_{\eta}},
\frac{\nabla^\bot u}{ u_{\eta}} \rangle = 0,  \\
& u(T,\xi,\eta) = \varphi(\xi, \eta)
\end{split}
\end{align}
with various choices of the ``diffusion constant'' $C$. We rely on the
existence of a sufficiently smooth solution with certain qualitative properties, namely
\eqref{require-u-smooth}--\eqref{require-time-deriv-neg}. We asserted in
Section \ref{sec2-5} the existence of such a solution under certain conditions upon
$\varphi$, namely \eqref{require-phi-smooth}--\eqref{require-cond-neg-time-deriv}. The main goal
of this appendix is to prove that assertion. We also briefly discuss a geometric
interpretation of the PDE, and we give an explicit formula for $u$ in the classic
case $\varphi = \half (\eta + |\xi|)$. All these results are already known \cite{Zhu},
but we present a self-contained treatment here for the reader's convenience.

\subsection{Motivation} \label{sec6-1}
We begin with the observation that if $\varphi(\xi,\eta)$ satisfies
\begin{equation} \label{phi-eta-pos}
\varphi_\eta \geq c > 0
\end{equation}
for some constant $c$, then for each $y \in \RR$ there is a unique
value $g(\xi;y)$ such that
\begin{equation} \label{phi-determines-g}
\varphi(\xi, g(\xi;y)) = y.
\end{equation}
Thus: for each $y$ the level set $\{ \varphi = y \}$ is the graph of a function
$\xi \mapsto g(\xi;y)$. These graphs foliate the $(\xi,\eta)$ plane and provide an alternative
representation of the function $\varphi$. Implicit differentiation in $\xi$ reveals that
$$
\varphi_\xi + \varphi_\eta g_\xi = 0 \quad \mbox{and} \quad
\varphi_{\xi \xi} + 2 \varphi_{\xi \eta} g_\xi + \varphi_{\eta \eta} g_\xi^2 + \varphi_\eta g_{\xi \xi} = 0
$$
when $\eta = g(\xi;y)$.
In view of \eqref{phi-eta-pos} we can solve for $g_\xi$ and $g_{\xi \xi}$, obtaining
\begin{equation} \label{derivs-of-g}
g_\xi = - \frac{1}{\varphi_\eta} \varphi_\xi \quad \mbox{and} \quad g_{\xi \xi} =
-\frac{1}{\varphi_\eta^3}
(\varphi_{\xi \xi}\varphi_\eta^2 - 2\varphi_{\xi \eta}\varphi_\xi \varphi_\eta +
\varphi_{\eta \eta}\varphi_\xi^2).
\end{equation}
There is also a simple formula for $g_y$: differentiation of \eqref{phi-determines-g} with respect to $y$
gives
\begin{equation} \label{g_y}
g_y = 1/\varphi_\eta.
\end{equation}

If we accept that the PDE \eqref{PDE-with-diffusion-const-C} has a solution
$u(t,\xi,\eta)$ with $u_\eta > c > 0$, then the preceding discussion applies to it as
well (treating the time $t$ as a parameter). So there is a function $h(t,\xi;y)$
such that
\begin{equation} \label{u-determines-h}
u(t,\xi, h(t,\xi;y)) = y
\end{equation}
for all $\xi$ and $y$, and all $t < T$. Differentiation in time gives
$$
u_t + u_\eta h_t = 0
$$
and differentiation with respect to $\xi$ gives the analogue of \eqref{derivs-of-g}:
\begin{equation} \label{derivs-of-h}
h_\xi = - \frac{1}{u_\eta} u_\xi \quad \mbox{and} \quad h_{\xi \xi} =
-\frac{1}{u_\eta^3}
(u_{\xi \xi}u_\eta^2 - 2 u_{\xi \eta} u_\xi u_\eta + u_{\eta \eta}u_\xi^2)
\end{equation}
when $\eta = h(t,\xi;y)$. Since
\begin{equation} \label{spatial-operator-explicit}
\langle D^2 u \nabla^\bot u, \nabla^\bot u \rangle =
u_{\xi \xi}u_\eta^2 - 2 u_{\xi \eta} u_\xi u_\eta + u_{\eta \eta}u_\xi^2,
\end{equation}
we see that \eqref{PDE-with-diffusion-const-C} holds exactly if $h(t,\xi;y)$ solves
the linear heat equation
\begin{equation} \label{h-eqn-with-constant-C}
h_t + \frac{1}{2} C h_{\xi \xi} = 0
\end{equation}
for all $y$ and $\xi$ and all $t < T$, with the final-time condition
\begin{equation} \label{h-eqn-final-value}
h(T, \xi; y) = g(\xi; y).
\end{equation}
The hypothesis that $u_\eta$ is positive is consistent with this PDE characterization of $h$: indeed, differentiating
\eqref{u-determines-h} with respect to $y$ gives
\begin{equation} \label{h_y}
u_\eta = 1 / h_y,
\end{equation}
and differentiation of \eqref{h-eqn-with-constant-C}--\eqref{h-eqn-final-value} reveals that
\begin{equation} \label{u-eta-lower-bound}
\mbox{$\varphi_\eta \geq c$ implies $u_\eta \geq c$}
\end{equation}
for any constant $c>0$, using the maximum principle for the linear heat equation and the relations
$u_\eta = 1/h_\eta$, $g_\eta = 1/\varphi_\eta$.

Summarizing: we have shown that if $u$ exists and is sufficiently smooth, then for each $y$ the
level set $\{ u=y \}$ is the graph of a function $h(t,\xi;y)$, where $h$ solves the
final-value problem \eqref{h-eqn-with-constant-C}--\eqref{h-eqn-final-value} in $t$ and $\xi$.

\subsection{Existence and properties of $u$}

Our construction of $u$ reverses the preceding discussion: we solve the linear heat equation to get $h$, then
use $h$ to get $u$.

\begin{theorem} \label{ZhuT}
Suppose $\varphi(\xi,\eta)$ is a $C^4$ function on $\RR^2$, with
\begin{equation} \label{phi-eta-pos-bis}
\varphi_\eta \geq c > 0.
\end{equation}
for some constant $c$. While $\varphi$ can (indeed, must) have linear growth at infinity, we assume
that $\varphi_\xi, \varphi_\eta$, and all higher derivatives of order up to $4$ are
bounded.
Then for any positive constant $C$ the final-value problem
\begin{align} \label{PDE-with-diffusion-const-C-bis}
\begin{split}
&u_t+\frac{1}{2} C \langle D^2 u \frac{\nabla^\bot u}{ u_{\eta}},
\frac{\nabla^\bot u}{ u_{\eta}} \rangle = 0 \ \mbox{for $t<T$}\\
& u(T,\xi,\eta) = \varphi(\xi, \eta)
\end{split}
\end{align}
has a classical solution $u(t,\xi,\eta)$ such that
\begin{gather}
u_\eta \geq c \label{u-eta-pos}\\
\mbox{$u$ has continuous, uniformly bounded derivatives in $(\xi, \eta)$ of order up to $4$}
\label{bounds-on-u} \\
\mbox{$u_t$ has continuous, uniformly bounded derivatives in $(\xi, \eta)$ of order up to $2$, and}
\label{bounds-on-ut}\\
\mbox{$u_{tt}$ is continuous and uniformly bounded,} \label{bound-on-utt}
\end{gather}
all bounds being uniform as $t \uparrow T$.
Moreover, the following structural properties of $\varphi$ persist to $u$:
\begin{gather}
\mbox{if $|\varphi_\xi| \leq \varphi_\eta$ for all $\xi,\eta$, then $|u_\xi| \leq u_\eta$ for all $\xi,\eta$
and all $t \leq T$;}
\label{property1}\\
\mbox{if $\varphi_{\xi \xi}\varphi_\eta^2 - 2\varphi_{\xi \eta}\varphi_\xi \varphi_\eta +
\varphi_{\eta \eta}\varphi_\xi^2 \geq 0$ for all $\xi,\eta$, then $u_t \leq 0$ for all
$\xi,\eta$ and all $t \leq T$;}
\label{property2}\\
\mbox{if $\varphi$ is odd in $\xi$ then so is $u$.}
\label{property3}
\end{gather}
\end{theorem}

\begin{proof}
We begin with the level-set representation of $\varphi$, i.e. with the function $g(\xi;y)$ defined by
\eqref{phi-determines-g}. To assess the smoothness of $g$ we apply the inverse function theorem to the
maps $G$ and $\Phi$ defined by
\begin{equation} \label{inverse-function-for-g}
(\xi, y) \stackrel{G}{\rightarrow} (\xi, g(\xi; y))  \quad \mbox{and} \quad
(\xi, \eta) \stackrel{\Phi}{\rightarrow} (\xi, \varphi(\xi, \eta)).
\end{equation}
The relation $\varphi(\xi,g(\xi;y)) = y$ says that $\Phi \circ G$ is the identity map, so $G$ is the
inverse of $\Phi$. The Jacobian of $\Phi$ is uniformly bounded and positive, by \eqref{phi-eta-pos-bis}. It follows,
by the inverse function theorem (see e.g. \cite{Dieu}) that $G$ has the same smoothness as $\Phi$; in
particular, since $\varphi$ has uniformly bounded derivatives of order up to $4$, so does $g$.

Now consider the solution $h(t,\xi; y)$ of the final-value problem \eqref{h-eqn-with-constant-C}--\eqref{h-eqn-final-value}. Differentiating the PDE in $\xi$ and/or $y$
and applying the maximum principle for the linear heat equation,
we see that the derivatives of $h$ in $\xi,y$ of order up to $4$ are uniformly bounded. Since
$g_y$ is uniformly positive as well as uniformly bounded (by \eqref{g_y}), the same is true of $h_y$.
So for each $t$, the graphs $\eta = h(t,\xi;y)$ foliate the $(\xi,\eta)$ plane as $y$ varies, determining
a unique function $u$ such that \eqref{u-determines-h} holds. The arguments in Section \ref{sec6-1} show that
$u$ is a classical solution of the PDE \eqref{PDE-with-diffusion-const-C-bis}, and that $u_\eta \geq c$. To assess
its regularity in $\xi$ and $\eta$, we observe that for each fixed $t$ the functions
$$
(\xi, y) \rightarrow (\xi, h(t,\xi; y))  \quad \mbox{and} \quad
(\xi, \eta) \rightarrow (\xi, u(t, \xi, \eta))
$$
are inverse to one another; so the regularity of $h$ implies, via the inverse function theorem, that
$u$ has uniformly bounded derivatives of order up to $4$ in $\xi$ and $\eta$. The regularity of $u$
in $t$ is best assessed using the PDE: since $\langle D^2 u \frac{\nabla^\bot u}{ u_{\eta}},
\frac{\nabla^\bot u}{ u_{\eta}} \rangle$ has bounded second spatial derivatives, so does $u_t$. It follows
immediately (by differentiating the PDE in $t$) that $u_{tt}$ is also uniformly bounded.

Turning now to qualitative properties, recall from \eqref{derivs-of-g} and \eqref{derivs-of-h} that
$|\varphi_\xi| \leq \varphi_\eta$ is equivalent to $|g_\xi| \leq 1$, and $|u_\xi| \leq u_\eta$ is equivalent
to $|h_\xi| \leq 1$. Since $|g_\xi| \leq 1$ implies $|h_\xi| \leq 1$ by the maximum principle for the linear
heat equation, \eqref{property1} is clear.

Next, recall from \eqref{derivs-of-g}--\eqref{spatial-operator-explicit} that
$ \varphi_{\xi \xi}\varphi_\eta^2 - 2\varphi_{\xi \eta}\varphi_\xi \varphi_\eta +
\varphi_{\eta \eta}\varphi_\xi^2 \geq 0$ is equivalent to $g_{\xi \xi} \leq 0$, while
$u_{\xi \xi} u_\eta^2 - 2 u_{\xi \eta} u_\xi u_\eta + u_{\eta \eta} u_\xi^2 \geq 0$ is
equivalent to both $h_{\xi \xi} \leq 0$ and $u_t \leq 0$. Since $h_{\xi \xi}$ solves a linear
heat equation with $g_{\xi \xi}$ as final-time data, \eqref{property2} follows once again from the
maximum principle for the linear heat equation.

Finally, recall from \eqref{phi-determines-g} and \eqref{u-determines-h} that
$\varphi(\xi,\eta) = \varphi(-\xi,\eta)$ is equivalent to $g(\xi) = g(-\xi)$, and
$u(\xi,\eta) = u(-\xi,\eta)$ is equivalent to $h(t,\xi,\eta) = h(t, -\xi, \eta)$. So
\eqref{property3} follows from the fact that the solution of a linear heat equation is
odd if the final-time data are odd.
\end{proof}

\subsection{Geometric interpretation of $u$}

We have shown that for each $y$, the evolution of the level set $u=y$ can be found without
considering any other level sets (by solving a linear heat equation). Second-order parabolic
PDE's with this property are called ``geometric,'' because the normal velocity of each level
set can be written in terms of its normal direction and curvature \cite{giga-book}.

It is natural to ask what our PDE \eqref{PDE-with-diffusion-const-C} looks like
from this perspective (though the main part of our paper makes no use of this result).
The answer is that the normal velocity of the
level set $\{ u = y \}$, viewed as a curve in the $(\xi,\eta)$ plane, is related to the level set's
unit normal $\vec{n}$ and curvature $\kappa$ by
\begin{equation} \label{v_nor}
v_{nor} = - \frac{C}{2} \frac{|\nabla u|^2}{|u_\eta|^2} \kappa =
- \frac{C}{2(\vec{n} \cdot (0,1))^2} \kappa.
\end{equation}
Indeed, the curvature of a level set
$$
\kappa = - \mbox{div} \left( \frac{\nabla u}{|\nabla u|} \right)
$$
has the property that
$$
\langle D^2 u \frac{\nabla^\bot u}{|\nabla u|},
\frac{\nabla^\bot u}{|\nabla u|} \rangle = -\kappa |\nabla u|,
$$
so the PDE \eqref{PDE-with-diffusion-const-C} says
$$
u_t = \frac{1}{2}C \kappa \frac{|\nabla u|^3}{u_\eta^2}.
$$
This leads directly to \eqref{v_nor}, since the velocity of the level set is (by definition)
$v_{nor}=  -u_t / |\nabla u| $ and the unit normal is $\vec{n} = (u_\xi,u_\eta)/|\nabla u|$.

\subsection{Explicit solution for the classic final-time data}

As we observed in the Introduction, in the classic case $\varphi = \frac{1}{2}(\eta + |\nabla \xi|)$ the
solution of \eqref{PDE-with-diffusion-const-C} has the form $u = \frac{1}{2} \eta + \overline{u}(t,\xi)$, where
$\overline{u}$ solves the linear heat equation $\overline{u}_{t} + \frac{1}{2}C \overline{u}_{\xi \xi} = 0$
for $t<T$ with $\overline{u}(T,\xi) = \frac{1}{2} |\xi|$. It is amusing to observe that this
function $\overline{u}$ has an explicit formula, namely
\begin{equation} \label{ansg}
\overline{u}(t, \xi) = \sqrt{T-t} \, G \left( \xi / \sqrt{T-t} \right)
\end{equation}
where
\begin{equation} \label{Geqn}
G(z) = \sqrt{\frac{C}{2\pi}} \exp \left( \frac{-z^2}{2C} \right) +
\frac{z}{2} \erf \left( \frac{z}{\sqrt{2C}} \right)
\end{equation}
(with the usual convention that $\erf(x) = \frac{2}{\sqrt{\pi}} \int_0^x  e^{-s^2} \, ds$,
so that $\erf(x) \rightarrow 1$ as $x \rightarrow \infty$).
Indeed, a brief calculation reveals that a function of the form \eqref{ansg} solves
the desired PDE exactly if $G$ solves
$$
G(z)-zG'(z)-C G''(z)=0,
$$
and it is easy to check that the proposed function $G$ (given by \eqref{Geqn}) meets this requirement.
Turning to the final-time behavior: it is easy to check that our $G$ satisfies
$$
\lim_{z \rightarrow \infty} \frac{G(z)}{z}=\frac{1}{2} \quad \mbox{and} \quad
\lim_{z \rightarrow -\infty} \frac{G(z)}{z}=-\frac{1}{2},
$$
and it follows that
$$
\lim_{t \rightarrow T} \overline{u}(t,\xi) = |\xi|/2
$$
as expected.

\bibliography{TimeDepStrategies}
\bibliographystyle{plain}

\end{sloppypar}
\end{document}